\documentclass{cpaaEA} 
\usepackage{amsmath}
\usepackage{paralist}
\usepackage[misc]{ifsym}
\usepackage{epsfig} 
\usepackage{epstopdf} 
\usepackage[colorlinks=true]{hyperref}
\usepackage{mathrsfs}
\hypersetup{urlcolor=blue, citecolor=red}
\allowdisplaybreaks

\def\currentvolume{X}
 \def\currentissue{X}
  \def\currentyear{200X}
   \def\currentmonth{XX}
    \def\ppages{X-XX}
     \def\DOI{10.3934/cpaa.2025061}

\newtheorem{theorem}{Theorem}[section]

\newtheorem{lemma}[theorem]{Lemma}

\theoremstyle{definition}

\newtheorem{remark}[theorem]{Remark}

\numberwithin{equation}{section}

\title[Nonlinear fractional Schr\"odinger system]
{Existence and non-existence of normalized solutions for  a  nonlinear fractional Schr\"odinger system} 

\author[Chungen Liu, Zhigao Zhang and JiaBin Zuo]{}


\thanks{The first  author was supported by the National Natural Science Foundation of China (Grant No.12171108). The third  author was supported by the Guangdong Basic and Applied Basic Research
	Foundation (2024A1515012389).
}

\thanks{$^*$Corresponding author: Jiabin Zuo}

\begin{document}
\maketitle
\centerline{\scshape
     Chungen Liu$^{{\href{liucg@nankai.edu.cn}{\textrm{\Letter}}}}$,
	 Zhigao Zhang$^{{\href{Zhangzhigao@e.gzhu.edu.cn}{\textrm{\Letter}}}}$ and
	 	 Jiabin Zuo$^{{\href{zuojiabin88@163.com}{\textrm{\Letter}}}*}$}

\medskip

{\footnotesize
	\centerline{School of Mathematics and Information Science, Guangzhou University,}
	
	 \centerline{Guangzhou, 510006, PR China}
} 


\bigskip

\centerline{(Communicated by Luc Nguyen)}


\begin{abstract}
This paper is concerned with a nonlinear fractional Sch\"ordinger system  in $\mathbb{{R}}$  with   intraspecies interactions $a_{i}>0 \ (i=1,2)$ and interspecies interactions $\beta \in\mathbb{{R}}$. We study this system  by solving an associated constrained minimization problem (i.e., $L^2-$norm constaints). Under certain assumptions on the trapping potentials $V_i(x) \ (i=1,2),$  we  derive some delicate estimates  for  the  related energy functional and establish a criterion for the existence and non-existence of solutions, in
which way several existence results are obtained.
\end{abstract}


\section{Introduction}\label{Intro}

	In this paper, we study the following system of two coupled fractional  Schr\"{o}dinger equations with trapping potentials in $\mathbb{R}$:\\
\begin{equation}\label{system1.1}
	\begin{cases}
		\sqrt{-\Delta }u_{1}+V_{1}(x)u_{1} =\mu_{1}u_{1}+a_{1} u_{1}^{3}+\beta u_{2}^{2}u_ {1},\\
		\sqrt{-\Delta }u_{2}+V_{2}(x)u_{2} =\mu_{2}u_{2}+a_{2} u_{2}^{3}+\beta u_{1}^{2}u_ {2},
	\end{cases}
\end{equation}\\
which $\sqrt{-\Delta }$ is the $\frac{1}{2}$-Laplace operator defined in the Sobolev space $H^{\frac{1}{2}}(\mathbb{R} )$ by the Fourier transform, that is,
\begin{equation}\label{four}
	\sqrt{-\Delta  } u:=\mathscr{F}^{-1}(\left | \xi  \right | \mathscr{F}u(\xi)),\ \ u\in H^{\frac{1}{2} }(\mathbb{R}),
\end{equation}
in which $	\mathscr{F}u(\xi )$ stands for the Fourier transform of $u$, i.e.,
\begin{equation*}
	\mathscr{F}u(\xi )=\frac{1}{\sqrt{2\pi}}\int_{\mathbb{R}}\exp^{- i\xi \cdot x}u(x)dx,\ \ u\in H^{\frac{1}{2} }(\mathbb{R})\ \ \text{and }
\end{equation*}
\begin{equation*}
	H^{\frac{1}{2} }(\mathbb{R})=\{ u\in L^{2} (\mathbb{R} ):\left \| u\right \|_{ H^{\frac{1}{2} }(\mathbb{R})}^{2}:=\int_{\mathbb{R}}(1+2\pi\left | \xi  \right | ) \Big| \mathscr{F}u(\xi)  \Big|^{2}d\xi< \infty   \} .
\end{equation*}
Moreover, for $i=1,2$, $V_{i}(x)$ is a certain type of trapping potential, $\mu_{i}\in \mathbb{R}$ is a suitable Lagrange multiplier which denotes the chemical potential, $a_{i}>0$ $(\text{resp}.<0)$ represents the attractive (resp. repulsive) intraspecies interaction strength inside each component, $\beta> 0$ $(\text{resp}. < 0)$ characterizes the attractive (resp. repulsive)   interspecies interaction strength   between two components.\\
In what follows,  the complex $L^{2}$ inner product is denoted by $(\cdot ,\cdot)$, i.e.,
\begin{equation*}
	(u,v):=\int_{\mathbb{R}}u\overline{v}dx, \ \  \forall u,v\in L^{2}(\mathbb{R},\mathbb{C}),
\end{equation*}
and then
\begin{equation}\label{exch}
	(\sqrt{-\Delta  }u,u ):=\int_{\mathbb{R}}\left | \xi  \right | \Big| \mathscr{F}u(\xi)  \Big|^{2}d\xi,\ \ \:  u\in
	H^{\frac{1}{2}}(\mathbb{R}).
\end{equation}
Moreover,  the norm  $\left \| \cdot  \right \| _{H^{\frac{1}{2}}(\mathbb{{R}})} $   can be given by
\begin{equation*}
	\left \| u  \right \| _{H^{\frac{1}{2}}(\mathbb{{R}})}^2 :=\int_{\mathbb{R}}\Big[u(-\Delta )^{\frac{1}{2} } u+u^{2}\Big]dx=\int_{\mathbb{R}}\Big[|(-\Delta )^{\frac{1}{4} } u|^{2}+u^{2}\Big]dx.
\end{equation*}

It is well-known that system (\ref{system1.1}) arises from the following one-dimensional fractional Schr\"{o}dinger equation in \cite{moxing},
\begin{equation}\label{ori}
	i\partial_t \psi(x,t) = (-\Delta)^{\alpha/2} \psi + V(x)\psi + \beta |\psi|^2 \psi, \quad x \in \mathbb{R},\ t > 0,
\end{equation}
where $\alpha \in(0,2)$ is the L\'{e}vy index, and $i=\sqrt{-1}$ is the imaginary unit. The constant $\beta\in\mathbb{R} $ describes the strength of local (or short-range) interactions between particles, and $V(x)$ represents an external trapping potential. $\psi(x,t)$ is a complex-valued wave function given by $\psi(x,t) = e^{-i\mu t} u(x)$, where $u(x)$ is a solution of the associated equation
\begin{equation}
	(-\Delta)^s u + V(x)u = \mu u + \beta u^3, \quad \text{in } \mathbb{R}.
\end{equation}
It is  known from \cite{moxing2}  that (\ref{ori}) is used to  describe the phenomenon in which the motion of bosons in non-Gaussian Bose-Einstein condensates follows Lévy flights instead of Gaussian paths. Furthermore, unlike the classical Laplacian, the fractional Laplacian is a non-local operator describing long-range interactions. It should be noted that, in the multi-component framework  (such as system (\ref{system1.1})), one has to take into account an interspecies interaction between each components, which may display more complicated physical phenomena, see \cite{He}. For more physical details, we refer the reader to \cite{add1, add2}.

In recent years, as a generalization of the classical Laplace operator, fractional Laplace operator holds broad application prospects in several areas such as physics, biology, chemistry, and finance, see  \cite{response1,response2}. Particularly in  Fourier analysis and nonlinear partial differential equations, the following fractional Schr\"{o}dinger equation (\ref{eq1.2}) with a nonlinear  term  has attracted much attention 
\begin{equation}\label{eq1.2}
	(-\Delta  )^{s}u+V(x)u=f(x,u), \ \  \: u\in H^{s}(\mathbb{R}^{N}).
\end{equation}

In \cite{x12}, under suitable assumptions for potential function $V(x)$ and nonlinear term $f(x,u)$, Felmer et al.  investigated the existence of  positive classical solutions for the system with $0<s<1$ and $N\ge2$ by using the critical point theory. It is noteworthy that, particularly in the case   $0< s< 1$ and $N=1$, through the variational method, Frank and Lenzmann \cite{x10} addressed the existence, uniqueness and a series of significant properties  of  the ground state on (\ref{eq1.2}) with $V(x)=1$ and $f(x,u)=u^{p+1}$, that is,
\begin{equation}\label{new classical}
	(-\Delta )^{s}u+u=u^{p+1},\ \ \: u\in H^{s}(\mathbb{{R}}).
\end{equation}
The famous work mentioned above and these important conclusions have opened up a new way for our investigations.
Later, in the general case  $0<s<1$ and  $N\ge 1$, the following equation
\begin{equation}\label{eq1.4}
	(-\Delta )^{s}u+u=|u|^{p}u,\ \ \: u\in H^{s}(\mathbb{{R}}^{N})
\end{equation}
was studied by \cite{x11}, from which we know that there is a unique  positive  solution $Q(x)\in H^{2s+1}(\mathbb{{R}}^{N})\cap C^{\infty }(\mathbb{{R}}^{N})$, which can be taken to be radially symmetric about the origin. Moreover,  $Q(x)$ is strictly decreasing in $|x-x_0|$,  and it satisfies
\begin{equation*}\label{eq1.5}
	\frac{C_{1} }{1+\left | x \right |^ {N+2s}} \le Q(x)\le \frac{C_{2}}{1+\left |x\right |^{N+2s}},\ \  \:x\in \mathbb{R}^{N}
\end{equation*}
with some constants $C_{2}\ge C_{1}>0$, which  generalizes the existential results and properties of solutions  in \cite{x10} in an optimal way.
In \cite{x9},  Du et al. considered a  mass critical stationary  fractional Schr\"odinger equation,  more  precisely, that is,
\begin{equation}\label{new}
	(-\Delta )^{s}u+V(x)u=\mu u+a|u|^{\frac{4s}{N}}u,\ \ x\in \mathbb{R}^{N},
\end{equation}
where $0 <s< 1$, $N\ge2$ and $\mu$ is a Lagrange multiplier. By applying constrained variational methods,  a complete classification of the existence and non-existence of $L^2$-normalized solutions  is established with certain conditions imposed on $V(x)$. For further related   results  on the fractional Schr\"{o}dinger equation, we  refer to \cite{x8,  x11, x21} and the references therein.

We know that the following classical Laplacian system in $\mathbb{{R}}^2$
\begin{equation}\label{system1.6}
	\begin{cases}
		-\Delta  u_{1}+V_{1}(x)u_{1} =\mu_{1}u_{1}+b_1 u_{1}^{3}+\beta u_{2}^{2}u_ {1},\\
		-\Delta  u_{2}+V_{2}(x)u_{2} =\mu_{2}u_{2}+b_2 u_{2}^{3}+\beta u_{1}^{2}u_ {2},
	\end{cases}
\end{equation}
which is used to model two-component Bose-Einstein condensates with  intraspecies and interspecies interactions. It is easy to see that when $u_{1}=0$ or $u_{2}=0$, or $u_{1}=u_{2}$, the system (\ref{system1.6}) can  be reduced to the single Schr\"{o}dinger equation (\ref{new}) with $s=1$ and $N=2$. In the past decade, extensive research has been conducted on the analogue issues associated with (\ref{system1.6}) for the classical Laplacian system, under variant conditions on $V_{i}(x)$, $b_{i}$ and $\beta$. In \cite{x14}, with $b_i>0$, $\beta>0$, and for a certain type of trapping potentials $V_{i}(x)$, Guo et al. investigated (\ref{system1.6})  by seeking minimizers of the associated variational functional $\tilde{e} (b_{1}, b_{2}, \beta)$ and then obtained the following  results in Theorem 1.1	\\

\noindent	(i). When $0<b_1<m^*$, $0<b_2<m^*$ and $\beta<\sqrt{(m^*-b_1)(m^*-b_2)}$, $\tilde{e} (b_{1}, b_{2}, \beta)$  has at
	least one minimizer. 

\noindent	(ii). Either $b_1>m^*$ or  $b_2>m^*$ or $\beta>\frac{m^*-b_1}{2}+\frac{m^*-b_2}{2}$, $\tilde{e} (b_{1}, b_{2}, \beta)$ has no minimizer.\\

\noindent Here, $m^*= \| \tilde{Q}  \| ^{2}_{2}$ is a critical threshold for existence intervals of solutions, where $\tilde{Q}$ is a positive solution to the following classical equation
			\begin{equation*}
		-\Delta u+u-u^3=0,\ \ u\in H^1(\mathbb{R}^2).
\end{equation*}
\noindent It is noteworthy that the existence of a positive solution to this classical equation is  necessary for establishing such existence intervals of solutions. For more details, see Theorems 1.1-1.3 in \cite{x14}.  	Later, in \cite{x15}, Guo, Zeng and Zhou generalized various existence results to the case $b_i>0$, $\beta<0$. For further details, see e.g. \cite{x1, x2, x4, x6,  x16, x18, x20}.

For the case of fractional Laplacian systems, in \cite{x23}, given a class of fractional elliptic systems of order $s\in(\frac{1}{4},1)$, Wang and Wei proved the uniqueness of the positive solutions, via the method of moving planes. In \cite{x17}, He et al. considered  a system with critical nonlinearities on a bounded domain in $\mathbb{R}^{N}$. With the help of  the Nehari manifold decomposition, they proved that the system admits at least two positive solutions when the pair of parameters $(\lambda,\mu)$ belongs to a certain subset of $\mathbb{R}^{2}$. In \cite{x7}, Du and Mao investigated a fractional Schr\"{o}dinger system with linear and nonlinear couplings, where $s\in{(\frac{3}{4},1)}$ and the potential function $V(x)$ satisfies $$\mathop{\mathrm{sup\mathrm{} } }\limits _{x\in \mathbb{R}^{3} }V(x)=\displaystyle\lim_{\lvert  x\rvert \rightarrow +\infty }V(x)=\Lambda >0 ~\text{and}~ \mathop{\mathrm{inf\mathrm{} } }\limits _{x\in \mathbb{R}^{3} }V(x)\ge 0,$$ the existence of ground state solutions for the fractional Laplacian systems was obtained by using variational methods. In \cite{x24}, Wang et al. demonstrated the existence of $L^{2}$-normalized solutions for a nonlinear pseudo-relativistic elliptic system involving the operator  $\sqrt{-\Delta  +m^{2}}$ within  certain ranges of the three parameters $a_1, a_2, \beta$.

Inspired by \cite{x14}, we first consider the following fractional Laplacian system in $\mathbb{R}^N$ 
\begin{equation}\label{sys3}
	\begin{cases}
		({-\Delta })^su_{1}+V_{1}(x)u_{1} =\mu_{1}u_{1}+a_{1} u_{1}^{3}+\beta u_{2}^{2}u_ {1}, \\
		({-\Delta })^su_{2}+V_{2}(x)u_{2} =\mu_{2}u_{2}+a_{2} u_{2}^{3}+\beta u_{1}^{2}u_ {2},
	\end{cases}
\end{equation} 
which corresponds to the $N$-dimensional classical equation (\ref{eq1.4}) with $p=2$, that is,
\begin{equation}\label{eqadd}
	(-\Delta )^{s}u+u=|u|^{2}u,\ \ \: u\in H^{s}(\mathbb{{R}}^{N}).
\end{equation}			
For system (\ref{sys3}), in order to establish the results  analogous to Theorems 1.1-1.3 in \cite{x14}, we  follow the approach in \cite{x14}  and then assume  that  there exists  a positive solution to (\ref{eqadd}).  Furthermore, as shown in Theorem 1.1 of  \cite{x9},  we find that (\ref{eqadd})  is a special case of the following equation,   
\begin{equation}\label{relation}
	\frac{N}{2s}(-\Delta)^su+\left(1+\frac{1}{2}\left(2-\frac{N}{s}\right)\right)u-|u|^2u=0,\quad u\in H^s(\mathbb{R}^N).
\end{equation}
A comparison of the coefficients in  (\ref{eqadd})  and  (\ref{relation}) shows that the parameters $s$ and $N$ can only satisfy either the conditions $s=1$ and $N=2$, or $s=\frac{1}{2}$ and $N=1$. This is because in order to determine the optimal constant $2/\left \| Q \right \| ^{2}_{2}$ in \eqref{eq1.10} generated by the contraction of the fractional Gagliardo-Nirenberg-Sobolev inequality, which is a very important indicator for the existence interval of solutions,  (\ref{eqadd})  should satisfy the form of \eqref{relation}. The former corresponds to system (\ref{system1.6}) has been examined in \cite{x14}. Therefore,  we will consider the  latter that is the system (\ref{system1.1}) in $\mathbb{R}$, which incorporates  the non-local $\frac{1}{2}$-Laplace operator  and the trapping potentials $V_{i}(x) \ (i=1,2)$ with  the following assumptions:\\

$(\mathcal{D}_{1})$ $V_{i}(x)\in L^{\infty } _{loc}(\mathbb{R}) $,  $\displaystyle\lim_{\lvert  x\rvert \rightarrow \infty }V_{i}(x)=\infty $  and  $\mathop{\mathrm{inf\mathrm{} } }\limits _{x\in \mathbb{R} }V_{i}(x)=0$, $i=1,2$.\ \

$(\mathcal{D}_{2})$ $\mathop{\mathrm{inf\mathrm{} } }\limits _{x\in \mathbb{R} }\big(V_{1}(x)+V_{2}(x)\big)$  and  $\mathop{\mathrm{inf\mathrm{} } }\limits _{x\in \mathbb{R} }V_{i}(x) \ (i=1,2)$  are   attained.\\

In order to obtain our results, we consider the following constrained minimization problem
\begin{equation}\label{problem 1.7}	
	\hat{e} (a_{1}, a_{2}, \beta):=\mathop{\mathrm{inf\mathrm{} } }\limits _{(u_{1}, u_{2})\in \mathcal{M}}E_{a_{1}, a_{2} , \beta}(u_{1}, u_{2}),
\end{equation}
where the corresponding variational function $	E_{a_{1}, a_{2} , \beta}(u_{1}, u_{2})$ is defined by
\begin{equation}\label{eq1.9}
\begin{split}
		E_{a_{1}, a_{2} , \beta}(u_{1}, u_{2})=&\sum_{i=1}^{2} \int_{\mathbb{R}} [u_{i}\sqrt{-\Delta  } u_{i}+V_{i}(x)u_{i}^{2}]dx - \sum_{i=1}^{2}\frac{a_{i}}{2}\int_{\mathbb{R}}\left | u_{i} \right |^{4}dx   \\
	&-\beta\int_{\mathbb{R}}  \left | u_{1} \right |^{2}\left | u_{2} \right |^{2} dx,
\end{split}
\end{equation}
which is defined on the space
\begin{equation*}
	\mathcal{X}=\mathcal{H}_{1}\times\mathcal{H}_{2},\ \  \mathcal{H}_{i}=\{ u\in H^{\frac{1}{2} }(\mathbb{R}): \int_{\mathbb{R} }V_{i}(x)u^{2}dx<\infty \}\ \ (i=1,2) \quad
\end{equation*}
with associated norm
\begin{equation}\label{eq1.8}
	\left \| u  \right \| _{\mathcal{H}_{i} }^{2}=\int_{\mathbb{R} }\left [  | (-\Delta  ) ^{\frac{1}{4} }u|^{2}+V_{i}(x)u^{2} \right ]dx, \ \  \left \| u \right \|_{\mathcal{X}}^{2}=\left \| u \right \|_{\mathcal{H}_{1}}^{2}+\left \| u \right \|_{\mathcal{H}_{2}}^{2},\ \  i=1,2.
\end{equation}
Here,  $\mathcal{M}$ is the so-called mass constraint, which is given by
\begin{equation*}
	\mathcal{M} =\left \{ (u,v)\in \mathcal{X}:\int_{\mathbb{R} }u^{2}dx = \int_{\mathbb{R} }v^{2}dx =1\right \} .
\end{equation*}
The solution found on $\mathcal{M}$ is called a normalized solution, and it represents a kind of mass conservation in physics, which makes our study more meaningful. Inspired by \cite{x14}, we  investigate  problem (\ref{problem 1.7}) by formulating an appropriate auxiliary minimization problem, based on which we can establish a criterion on the existence of minimizers for the problem (\ref{problem 1.7}), as detailed in our Theorem \ref{theorem 2.4}, which plays an important role in proving the main results of this paper. However, our paper will take different approaches as in  \cite{ADD} to give a  proof of Theorem \ref{theorem 2.4}, and note that in \cite{x14} it only considered the case of $\beta>0$. Meanwhile, the presence of the non-local operator $\sqrt{-\Delta }$ presents a complex challenge in establishing this criterion, the difficulty of which is the selection of test functions and delicate estimates, see (\ref{eq2.48})-(\ref{eq2.49}).
Furthermore,  the approaches  employed in papers \cite{x14} do not work   in proving  Theorem \ref{theorem1.3},  as it remains uncertain whether we can derive the same sequence $\left \{ y_{\epsilon _{n}}\right \} $ from $\left \{ (\tilde{w}_{1n}  ,\tilde{w}_{2n}) \right \} $, where $\{\tilde{w}_{1n}\}$ satisfies (\ref{eq3.27})  and $\{\tilde{w}_{2n}\}$ satisfies (\ref{eq new}). Therefore, we borrow some techniques from \cite{x24} (for pseudo-relativistic operator systems) to overcome this difficulty.

We now define $a^{*}:=\left \| Q \right \|_{2}^{2}$, that is,  the square of the $L^2$-norm of the unique  positive  radially symmetric solution (\cite{x10,x11}) of the following   classical equation
\begin{equation}\label{classical eq}
	\sqrt{-\Delta }u+u=u^{3},\ \ \:u\in H^{\frac{1}{2}}(\mathbb{{R}}).
\end{equation}
Moreover, it is shown in Proposition 3.1 of \cite{x11} that every positive solution of (\ref{classical eq}) satisfies the following decay rate
\begin{equation}\label{decay}
	Q(x)=O(|x|^{-2})\ \ \text{as}\ \  |x|\to \infty.
\end{equation}
From Theorem 1.1 of \cite{x9}, we  know that the Gagliardo-Nirenberg inequality, which corresponds to (\ref{classical eq}), satisfies the following type
\begin{equation}\label{eq1.10}
	\int_{\mathbb{R}}\left | u \right |^{4}dx\le \frac{2}{\left \| Q \right \| ^{2}_{2}} \int_{\mathbb{R}}| \left ( -\Delta   \right )^{\frac{1}{4}}u   |^{2}dx    \int_{\mathbb{R}}\left | u \right |^{2}dx,\ \ \:u\in H^{\frac{1}{2}}(\mathbb{{R}}).
\end{equation}
Besides, by  Lemma 5.4 of \cite{x10},  $Q(x)$ satisfies the following Pohozaev
identity
\begin{equation}\label{eq1.11}
	|| \left ( -\Delta   \right )^{\frac{1}{4} } Q ||_{2}^{2}= \left \| Q \right \|_{2}^{2}=\frac{1}{2}\int_{\mathbb{R}} |Q|^{4}dx  .
\end{equation}

Now, we introduce the following theorem concerning the existence and non-existence of minimizers of problem (\ref{problem 1.7}).
\begin{theorem} \label{theorem1.1}
	Let $Q(x)$ be the unique positive radially symmetric solution of  (\ref{classical eq}) and $a^{*}:=\left \| Q \right \|_{2}^{2}$. Suppose that  $(\mathcal{D}_{1})$ holds, and denote
	\begin{equation}\label{beta}
		\beta_{*}:=\sqrt{(a^{*}-a_{1})(a^{*}-a_{2})} \ \ and \ \ \beta^{*}:=\frac{a^{*}-a_{1}}{2} +\frac{a^{*}-a_{2}}{2}.
	\end{equation}
	Then we have:
	
	\noindent\textnormal{(i)} If $0<a_{1}, a_{2}<a^{*}$ and $\beta<\beta_{*}$, there exists a minimizer for  problem (\ref{problem 1.7}).\\
	\noindent\textnormal{(ii)} If $a_{1}>a^{*}$, or $a_{2}>a^{*}$, or $\beta>\beta^{*}$, the problem (\ref{problem 1.7}) has no minimizer.
\end{theorem}

\begin{remark}
	From the definition of $\beta^+$ in (\ref{eq2.2}) and the proof of Theorem \ref{theorem1.1}, it is evident to see  that $\beta\le0$  is allowed. 	
\end{remark}
\begin{remark}
It is an open problem to determine a  criterion on the parameters ($a_1, a_2, \beta$)  which ensures the existence of minimizers for $s\neq \frac{1}{2}$.
\end{remark}
\begin{remark}
		It remains unknown whether there are other thresholds except  $a^*$, $\beta_{*}$ and  $\beta^{*}$ to ensure the existence of  minimizers.
\end{remark}

\begin{theorem}\label{theorem1.2}
	Suppose that   $(\mathcal{D}_{1})$ and $(\mathcal{D}_{2})$ hold, if \ $0<a_{1}\neq a_{2}<a^{*}$ satisfies
	\begin{equation*}
		\left | a_{1}-a_{2} \right | \le 2\sqrt{(a^{*}-a_{1})(a^{*}-a_{2})}=2\beta _{*},
	\end{equation*}
	then, there exists $\hat{\delta}\in (\beta_{*},\beta^{*}]$ such that problem (\ref{problem 1.7}) has a minimizer for any $\beta\in [\beta_{*}, \hat{\delta})$.
\end{theorem}

\begin{theorem}\label{theorem1.3}
	Suppose that   $(\mathcal{D}_{1})$ and $(\mathcal{D}_{2})$ hold,  if \  $0 < a_{1}= a_{2}<a^* $ and \ $\beta_{*}\le\beta\le\beta^{*}$ , then\\
	
	\noindent \textnormal{(i)} Problem (\ref{problem 1.7}) has no minimizer  if \ $\mathop{\mathrm{inf\mathrm{} } }\limits _{x\in \mathbb{R} }\big(V_{1}(x)+V_{2}(x)\big)=0$. \\
	\noindent\textnormal{(ii)} Problem (\ref{problem 1.7}) has at least a minimizer provided
	\begin{equation}\label{eq1.12}
		\hat{e} (a^{*}-\beta ,a^{*}-\beta,\beta )<\mathop{\mathrm{inf\mathrm{} } }\limits _{x\in \mathbb{R} }\big(V_{1}(x)+V_{2}(x)\big).
	\end{equation}
\end{theorem}

The structure of this paper is organized as follows. In Section \ref{section2}, we introduce an auxiliary minimization problem and derive the crucial Theorem \ref{theorem 2.4}. In Section \ref{section3}, we present the complete proofs of Theorems \ref{theorem1.1}, \ref{theorem1.2}, and \ref{theorem1.3}.

\section{Existence and non-existence of minimizers}  \label{section2}
In this section, we introduce an auxiliary minimization problem and then  address the proof of Theorem \ref{theorem 2.4}, which concerns the existence and nonexistence of minimizers for $	\hat{e} (a_{1}, a_{2}, \beta)$. To this end,  we first give a compactness result.

\begin{lemma}\label{lemma2.1}
	(\cite[Lemma 3.1]{x9}) Suppose  \  $0\le V_{i}(x)\in L^{\infty } _{loc}(\mathbb{R}) $ \ satisfies  $\displaystyle\lim_{\lvert  x\rvert \rightarrow \infty }V_{i}(x)\\=\infty $, \  $i=1,2$. Then for all $2\le q < \infty$, the embedding $\mathcal{H}_{1}\times\mathcal{H}_{2}\hookrightarrow L^{q}(\mathbb{R} )\times L^{q}(\mathbb{R} )$ is compact.
\end{lemma}

\begin{lemma}\label{lemma limit}
	If $(u_{1n},u_{2n})\to(u_{1},u_{2})$ strongly in  $L^{q}(\mathbb{R})\times L^{q}(\mathbb{R})$ \  for  $q\in[2,\infty)$, then
	\begin{align*}
		\lim_{n\to\infty}\int_{\mathbb{R}}|u_{1n}|^{2}|u_{2n}|^{2}dx=\int_{\mathbb{R}}|u_{1}|^{2}|u_{2}|^{2}dx,\\
		\lim_{n \to \infty} \int_{\mathbb{R}}|u_{in}|^{4}dx= \int_{\mathbb{R}}|u_{i}|^{4}dx,\ \ i=1.2.
	\end{align*}
	\begin{proof} Since $(u_{1n},u_{2n})$ is bounded  in $L^{q}(\mathbb{R}) \times L^{q}(\mathbb{R})$ for  $q\in[2,\infty)$, by the H\"older inequality, we have
		\begin{align}
			&\Big|\int_{\mathbb{R}}|u_{1n}|^2|u_{2n}|^2dx-\int_{\mathbb{R}}|u_1|^2|u_2|^2dx\Big|\nonumber\notag\\   \le&\Big|\int_{\mathbb{R}}|u_{1n}|^2(|u_{2n}|^2-|u_2|^2)dx\Big|+\Big|\int_{\mathbb{R}}(|u_{1n}|^2-|u_1|^2)|u_2|^2dx\Big|\nonumber\notag\\   \leq &\Big(\!\int_{\mathbb{R}}|u_{1n}|^4dx\!\Big)^{\frac{1}{2}}\Big(\!\int_{\mathbb{R}}(|u_{2n}|^2\!- \! |u_2|^2)^2dx\!\Big)^{\frac{1}{2}}\!+ \!\Big(\!\int_{\mathbb{R}}(|u_{1n}|^2\! - \!|u_1|^2)^2dx\!\Big)^{\frac{1}{2}}\Big(\int_{\mathbb{R}}|u_{2}|^4dx\Big)^{\frac{1}{2}}\notag\\   \leq  &\Big(\int_{\mathbb{R}}|u_{1n}|^4dx\Big)^{\frac{1}{2}}\Big(\int_{\mathbb{R}}(|u_{2n}|-|u_2|)^4dx\Big)^{\frac{1}{4}}\Big(\int_{\mathbb{R}}(|u_{2n}|+|u_2|)^4dx\Big)^{\frac{1}{4}}\label{eq 2.1.}\\ &+\Big(\int_{\mathbb{R}}(|u_{1n}|-|u_1|)^4dx\Big)^{\frac{1}{4}}\Big(\int_{\mathbb{R}}(|u_{1n}|+|u_1|)^4dx\Big)^{\frac{1}{4}}\Big(\!\int_{\mathbb{R}}|u_{2}|^4dx\!\Big)^{\frac{1}{2}}\notag\\ \le &C_{10}||u_{2n}-u_2||_{4}+C_{20}||u_{1n}-u_1||_{4}.\notag
		\end{align}

		The assertion follows by letting $n\to \infty$ in (\ref{eq 2.1.}). Similarly, the remaining proof can be obtained.
	\end{proof}
\end{lemma}
Then we introduce the following auxiliary minimization problem, for $a_{1}, a_{2}>0$ and $\beta\in \mathbb{{R}}$,
\begin{equation}\label{problem2.1}
	\Gamma(a_{1},a_{2},\beta )  := \inf \left \{ J_{a_{1},a_{2},\beta}(u_{1},u_{2}):u_{i}\in H^{\frac{1}{2}}(\mathbb{R} ),\left \| u_{i} \right \|^{2}_{2}=1, i=1,2  \right \},
\end{equation}
where
\begin{equation}\label{eq2.2}
\begin{split}
		J_{a_{1},a_{2},\beta}(u_{1},u_{2})&=\frac{2|| (-\Delta  )^{\frac{1}{4} }u_{1} ||^{2}_{2}+2\ || (-\Delta  )^{\frac{1}{4} }u_{2} ||^{2}_{2} }{a_{1}\int_{\mathbb{R}} \left | u_{1} \right |^{4}dx+a_{2}\int_{\mathbb{R}} \left | u_{2} \right |^{4}dx+2\beta^{+}\int_{\mathbb{R}} \left | u_{1} \right |^{2}\left | u_{2} \right |^{2}dx}, \\
	&\quad \quad\quad \quad\quad \quad\quad \quad\beta ^{+}=\max \left \{ 0,\beta  \right \}.
\end{split}
\end{equation}\\
Next, we present the following properties of $	\Gamma(a_{1},a_{2},\beta )  $. In this paper we set $\mathbb{R}_+=(0,+\infty)$.

\begin{lemma}\label{lemma2.2}
	Let $\Gamma(a_{1},a_{2},\beta )  $ be defined by (\ref{problem2.1}), then  $\Gamma(a_{1},a_{2},\beta )  $  is locally Lipschitz continuous with respect to $(a_{1},a_{2},\beta ) \in \mathbb{{R}}^{3}_{+}\setminus \left \{ (0,0,0)\right \}$.
\end{lemma}

\begin{proof}
	We first show that
	\begin{equation}\label{eq2.4}
		\frac{a^{*}}{\max \left \{ a_{1}+\beta, a_{2}+\beta \right \} }\le \Gamma(a_{1},a_{2},\beta )   \le \frac{2a^{*}}{a_{1}+a_{2}+2\beta}.
	\end{equation}
	For $Q(x)$ being given by (\ref{eq1.11}), we take $(u_{1},u_{2})=\Big(\frac{Q(x)}{||Q||_{2}},\frac{Q(x)}{||Q||_{2}}\Big)$ as a trial function of $\Gamma(a_{1},a_{2},\beta )$,  then the upper bound of (\ref{eq2.4}) can be derived from (\ref{eq1.11}) and (\ref{problem2.1}).
	In addition, for any $(u_{1},u_{2})\in H^{\frac{1}{2}}(\mathbb{{R}})\times H^{\frac{1}{2}}(\mathbb{{R})}$  satisfying $\int_{\mathbb{R}} \left | u_{i} \right |^{2}dx=1$ for $i=1, 2$, by applying  the  Gagliardo–Nirenberg inequality (\ref{eq1.10}), we have
	\begin{align*}
		J_{a_{1},a_{2},\beta}(u_{1},u_{2})&=\frac{2|| (-\Delta  )^{\frac{1}{4} }u_{1} ||^{2}_{2}+2\ || (-\Delta  )^{\frac{1}{4} }u_{2} ||^{2}_{2} }{a_{1}\int_{\mathbb{R}} \left | u_{1} \right |^{4}dx+a_{2}\int_{\mathbb{R}} \left | u_{2} \right |^{4}dx+2\beta\int_{\mathbb{R}} \left | u_{1} \right |^{2}\left | u_{2} \right |^{2}dx} \\
		&\ge \frac{a^{*}\Big(\int_{\mathbb{R}} \left | u_{1} \right |^{4}dx+\int_{\mathbb{R}} \left | u_{2} \right |^{4}dx\Big)}{\big(a_{1}+\beta \big)\int_{\mathbb{R}} \left | u_{1} \right |^{4}dx+\big(a_{2}+\beta \big)\int_{\mathbb{R}} \left | u_{2} \right |^{4}dx} \\
		&\ge \frac{a^{*}}{\max \left \{ a_{1}+\beta, a_{2}+\beta \right \} },
	\end{align*}
	which implies that the lower bound of (\ref{eq2.4}) is proved. This completes the proof of (\ref{eq2.4}).
	
	Let $\left \{ (u_{1n},u_{2n}) \right \} \subset H^{\frac{1}{2}}(\mathbb{{R}})\times H^{\frac{1}{2}}(\mathbb{{R})} $ be a minimizing sequence of $\Gamma(a_{1},a_{2},\beta )  $. Since (\ref{eq2.2}) is invariant under the rescaling: $u(x)\longmapsto \lambda^{\frac{1}{2} } u(\lambda x),  \lambda>0$, one may  suppose that
	\begin{equation}\label{eq2.6}
		|| (-\Delta  )^{\frac{1}{4} }u_{1n} ||^{2}_{2}+\ || (-\Delta  )^{\frac{1}{4} }u_{2n} ||^{2}_{2}=1 \ \ \text{for all} \ \ n\in \mathbb{N}.
	\end{equation}
	It then follows from the Gagliardo-Nirenberg inequality (\ref{eq1.10}) that
	\begin{equation*}
		\int_{\mathbb{R}} \left | u_{in} \right |^{4}dx\le \frac{2}{a^{*}} || (-\Delta  )^{\frac{1}{4} }u_{in} ||^{2}_{2}\le\frac{2}{a^{*}}, \ \ i=1,2,
	\end{equation*}
	and
	\begin{align*}
		\int_{\mathbb{R}} \left | u_{1n} \right |^{2}\left | u_{2n} \right |^{2}dx&\le \frac{1}{2} \Big( \int_{\mathbb{R}} \left | u_{1n} \right |^{4}dx+\int_{\mathbb{R}} \left | u_{2n} \right |^{4}dx \Big)  \\
		&\le \frac{1}{a^{*}}\big(|| (-\Delta  )^{\frac{1}{4} }u_{1n} ||^{2}_{2}+\ || (-\Delta  )^{\frac{1}{4} }u_{2n} ||^{2}_{2}\big)=\frac{1}{a^{*}} .
	\end{align*}
	Since $\left \{ (u_{1n},u_{2n}) \right \}$ is a minimizing sequence of $\Gamma(a_{1},a_{2},\beta ) $, by (\ref{eq2.6}), we infer that, for any $(a_{1},a_{2},\beta), (\hat{a}_{1},\hat{a}_{2},\hat{\beta}) \in \mathbb{{R}}^{3}_{+}\setminus \left \{ (0,0,0)\right \}$,
	\begin{align*}
		\frac{1}{\Gamma(a_{1},a_{2},\beta)}
		& =\lim_{n\to\infty}\Bigg[\frac{\hat{a}_{1} \int_{\mathbb{R}}  | u_{1n}  |^{4}dx+\hat{a}_{2} \int_{\mathbb{R}} \left | u_{2n} \right |^{4}dx+2\hat{\beta } \int_{\mathbb{R}} \left | u_{1n} \right |^{2}\left | u_{2n} \right |^{2}dx}{2\|(-\Delta)^{\frac{1}{4}}u_{1n}\|_{2}^{2}+2\|(-\Delta)^{\frac{1}{4}}u_{2n}\|_{2}^{2}} \\
		& \quad +\sum_{i=1}^2\frac{a_i-\hat{a}_i}2\int_{\mathbb{R}} \left | u_{in} \right |^{4}dx+(\beta-\hat{\beta})\int_{\mathbb{R}} \left | u_{1n} \right |^{2}\left | u_{2n} \right |^{2}dx \Bigg] \\
		&\leq\lim_{n\to\infty}\Bigg[\frac{1}{\Gamma(\hat{a}_{1},\hat{a}_{2},\hat{\beta})}+\sum_{i=1}^{2}\frac{|a_{i}-\hat{a}_{i}|}{2}\int_{\mathbb{R}} \left | u_{in} \right |^{4}dx\\
		&\quad+|\beta-\hat{\beta}|\int_{\mathbb{R}} \left | u_{1n} \right |^{2}\left | u_{2n} \right |^{2}dx\Bigg]\\
		&\leq\frac{1}{\Gamma(\hat{a}_{1},\hat{a}_{2},\hat{\beta})}+\sum_{i=1}^{2}\frac{|a_{i}-\hat{a}_{i}|}{a^{*}}+\frac{|\beta-\hat{\beta}|}{a^{*}},
	\end{align*}
	so there holds
	\begin{align}\label{eq2.7}
		\frac{1}{\Gamma(a_{1},a_{2},\beta)}-\frac{1}{\Gamma(\hat{a}_{1},\hat{a}_{2},\hat{\beta})}\leq\frac{3}{a^{*}}\big|(a_{1},a_{2},\beta)-(\hat{a}_{1},\hat{a}_{2},\hat{\beta})\big|.
	\end{align}
	Taking $\{(\tilde{u}_{1n},\tilde{u}_{2n})\}$ as  a minimizing sequence of $\Gamma(\hat{a}_{1},\hat{a}_{2},\hat{\beta})$, similar to the argument presented above, we further obtain that
	\begin{align}\label{eq2.8}
		\frac1{\Gamma(\hat{a}_1,\hat{a}_2,\hat{\beta})}-\frac1{\Gamma(a_1,a_2,\beta)}\leq\frac3{a^*}\big|(a_1,a_2,\beta)-(\hat{a}_1,\hat{a}_2,\hat{\beta})\big|,
	\end{align}
	which together with (\ref{eq2.7}) yields that
	\begin{equation}\label{eq eq}
\begin{split}
		\Big|\frac{\Gamma(a_1,a_2,\beta)-\Gamma(\hat{a}_1,\hat{a}_2,\hat{\beta})}{\Gamma(a_1,a_2,\beta)\Gamma(\hat{a}_1,\hat{a}_2,\hat{\beta})}\Big|&=\Big|\frac{1}{\Gamma(a_1,a_2,\beta)}-\frac{1}{\Gamma(\hat{a}_1,\hat{a}_2,\hat{\beta})}\Big|\\ 
	&\leq\frac{3}{a^*}\Big|(a_1,a_2,\beta)-(\hat{a}_1,\hat{a}_2,\hat{\beta})\Big|.
\end{split}
	\end{equation}
	By (\ref{eq2.4}) and  (\ref{eq eq}), a direct calculation gives that
	\begin{align}\label{ab ineq}
		\Big|\Gamma(a_{1},a_{2},\beta)\! -  \! \Gamma(\hat{a}_{1},\hat{a}_{2},\hat{\beta})\Big|\leq\frac{12a^{*}\Big|(a_{1},a_{2},\beta)\!- \!(\hat{a}_{1},\hat{a}_{2},\hat{\beta})\Big|}{(a_{1}+a_{2}+2\beta)(\hat{a}_{1}+\hat{a}_{2}+2\hat{\beta})},
	\end{align}
	which shows that $\Gamma(\cdot)$ is locally Lipschitz continuous in $\mathbb{R} _{+} ^{3} \setminus \{ ( 0, 0, 0) \}$.\\

	\noindent \textbf{Remark 2.1}.	It is noteworthy  that Lemma \ref{lemma2.2}  also holds for $\left(a_{1},a_{2}\right)\in\mathbb{R}_{+}^{2}\setminus\left\{\left(0,0\right)\right\}$ and $\beta\in \mathbb{R}$,    since $\Gamma(a_1,a_2,\beta)\equiv\Gamma(a_1,a_2,0)$  for $\beta\leq 0$.\\
	
	We recall that in this paper, $a^*=\|Q\|_2^2$ and the numbers $\beta_*$ and $\beta^*$ are defined in (\ref{beta}).
	\begin{lemma}\label{lemma2.3}
		Assume that $0<a_1\neq a_2 <  a^* $, then\\
		\quad $\textnormal{(i)}$ $\Gamma ( a_{1}, a_{2}, \beta _{* }) > 1$ $if$ $| a_{1}- a_{2}| \leq 2\beta _{* }$;\\
		\quad $ \textnormal{(ii)}$ $\Gamma \left ( a_{1}, a_{2}, \beta ^{* }\right ) \leq 1$.
		
	\end{lemma}
	\begin{proof}
		\renewcommand{\qedsymbol}{}
		(i) Let $\left \{ ( u_{1n}, u_{2n}) \right \} \subset \mathcal{M}$ be a minimizing sequence of $\Gamma(a_1,a_2,\beta_*)$. According to  the  property of Schwartz symmetrization shown in Theorem 3.7 in \cite{x19}, one may assume that $\left \{ ( u_{1n}, u_{2n}) \right \} \subset H^{\frac{1}{2}}_r(\mathbb{{R}})\times H^{\frac{1}{2}}_r(\mathbb{{R}})  $, i.e.,
		\begin{equation*}
			u_{in}(x)=u_{in}(|x|),\ \ i=1,2,.
		\end{equation*}
		Since (\ref{eq2.2}) is invariant under the rescaling: $u(x)\longmapsto \lambda^{\frac{1}{2} } u(\lambda x),  \lambda>0$, one  can   suppose that
		\begin{equation}\label{eq2.9}
			\|(-\Delta)^{\frac{1}{4}}u_{1n}\|_2^2+\|(-\Delta)^{\frac{1}{4}}u_{2n}\|_2^2=1\ \ \text{ for all }n\in\mathbb{N}^+.
		\end{equation}
		Combining (\ref{eq2.4}) and (\ref{eq2.9})  with  the  Gagliardo-Nirenberg inequality (\ref{eq1.10}), we can conclude that $\int_{\mathbb{R}} \left | u_{in} \right |^{4}dx \ (i=1,2)$ is bounded uniformly, that is,
		\begin{equation}\label{eq2.10}
			0<\hat{C}_1\leq a_1\int_{\mathbb{R}} \left | u_{1n} \right |^{4}dx+a_2\int_{\mathbb{R}} \left | u_{2n} \right |^{4}dx+2\beta\int_{\mathbb{R}} \left | u_{1n} \right |^{2}\left | u_{2n} \right |^{2}dx\leq\hat{C}_2<\infty.
		\end{equation}
		
		\noindent We  now  prove (i)  by analyzing the  following two cases: \\
		
		\noindent \textbf{Case 1.} If
		\begin{equation*}
			\int_{\mathbb{R}} \left | u_{1n} \right |^{4}dx\to 0,\ \text{or} \int_{\mathbb{R}} \left | u_{2n} \right |^{4}dx\to 0 \ \ \text{as} \ \ n\to \infty,
		\end{equation*}
		without loss of generality, we suppose that
		\begin{equation*}
			\int_{\mathbb{R}} \left | u_{1n} \right |^{4}dx\to 0 \ \ \text{as}\ \ n\to \infty .
		\end{equation*}
		It then follows from (\ref{eq2.10}) and the H\"older inequality that
		\begin{equation*}
			\int_{\mathbb{R}}|u_{1n}|^{2}|u_{2n}|^{2}dx\leq\Big(\!\int_{\mathbb{R}}|u_{1n}|^{4}dx\!\Big)^{\frac{1}{2}}\Big(\!\int_{\mathbb{R}}|u_{2n}|^{2}dx\!\Big)^{\frac{1}{2}} \overset{n}{\to} 0\  \text{and} \  \int_{\mathbb{R}} \left | u_{2n} \right |^{4}dx\geq C>0.
		\end{equation*}
		Since $0<a_2<a^*$, by the  Gagliardo-Nirenberg inequality (\ref{eq1.10}), we  have
		\begin{equation}
\begin{split}
				\Gamma(a_{1},a_{2},\beta_{*})&=\lim_{n\to\infty}\frac{	2|| (-\Delta  )^{\frac{1}{4} }u_{1n} ||^{2}_{2}+\ 2|| (-\Delta  )^{\frac{1}{4} }u_{2n}  ||^{2}_{2}}{a_{2}\int_{\mathbb{R}}|u_{2n}|^{4}dx+o(1)}\\
	&\geq\lim_{n\to\infty}\frac{a^{*}\int_{\mathbb{R}}|u_{2n}|^{4}dx}{a_{2}\int_{\mathbb{R}}|u_{2n}|^{4}dx+o(1)}\geq\frac{a^{*}}{a_{2}}>1.
\end{split}
		\end{equation}
		\noindent Thus, (i) is proved.\\
		
		\noindent \textbf{Case 2.} If
		\begin{equation*}
			\int_{\mathbb{R}} \left | u_{1n} \right |^{4}dx\geq C>0\ \ \text{and}\ \ \int_{\mathbb{R}} \left | u_{2n} \right |^{4}dx\geq C>0,
		\end{equation*}
		then it follows from (\ref{eq1.10}), (\ref{eq2.9}) and (\ref{eq2.10}) that  there are positive constants $C_{3}$ and $C_{4}$, independent of $n$, such that
		\begin{equation}\label{eq2.12}
			C_3\leq\|(-\Delta)^{\frac14}u_{1n}\|_2^2,\  \|(-\Delta)^{\frac14}u_{2n}\|_2^2,\  \int_{\mathbb{R}}|u_{1n}|^{4}dx,\  \int_{\mathbb{R}}|u_{2n}|^{4}dx\leq{C}_4.
		\end{equation}
		\noindent Since $\left \{ u_{in} \right \} \subset H^{\frac{1}{2}}_r(\mathbb{{R}})$  $(i=1, 2)$, it can be inferred from  (\ref{eq2.12}) and Lemma \ref{lemma2.1} that  there exists  $u_{i}(x)\subset H^{\frac{1}{2}}_r(\mathbb{{R}})$ such that, for $i=1,2$,
		\begin{equation}\label{eq2.13}
			u_{in}\overset{n}{\rightharpoonup }u_i \text{ weakly in }H^{\frac12}_r(\mathbb{R}),\  u_{in}\overset{n}{\to}u_i\text{ strongly in }L^p(\mathbb{R}), \ \ p\in(2,\infty).
		\end{equation}
		Then, it follows from  Lemma \ref{lemma limit} that
		\begin{equation*}
			\lim_{n \to \infty} \int_{\mathbb{R}}|u_{in}|^{4}dx= \int_{\mathbb{R}}|u_{i}|^{4}dx,\ \ i=1.2.
		\end{equation*}
		By (\ref{eq2.12}), we have
		\begin{equation}\label{eq2.14}
			\int_{\mathbb{R}}|u_{i}|^{4}dx\geq C>0\ \ \mathrm{and}\ \ u_{i}\not\equiv 0,\ \ i=1,2.
		\end{equation}
		
		Given that  $0<a_1\neq a_2<a^*$, without loss of generality, we may suppose that $a_1<a_2$, which, together with  the condition $|a_1-a_2|\leq2\beta_*=2\sqrt{(a^*-a_1)(a^*-a_2)}$ gives that
		\begin{equation}\label{eq2.15}
			0<a_1<a_2<a^* \mathrm{~~and~~}a_2\leq2\beta_*+a_1,\ \beta_*=\sqrt{(a^*-a_1)(a^*-a_2)}>0.
		\end{equation}
		\noindent We  deduce  from the Gagliardo-Nirenberg inequality (\ref{eq1.10}), (\ref{eq2.13}) and (\ref{eq2.14})   that
		\begin{align}
			\Gamma(a_{1},a_{2},\beta_{*})& =\lim_{n\rightarrow\infty}J_{a_{1},a_{2},\beta_{*}}(u_{1n},u_{2n}) \notag\\
			&\geq\lim_{n\to\infty}\frac{a^{*}\big(\int_{\mathbb{R}}|u_{1n}|^{4}dx+\int_{\mathbb{R}}|u_{2n}|^{4}dx\big)}{a_{1}\int_{\mathbb{R}}|u_{1n}|^{4}dx+a_{2}\int_{\mathbb{R}}|u_{2n}|^{4}dx+2\beta_{*}\int_{R}|u_{1n}|^{2}|u_{2n}|^{2}dx} \notag\\
			&=\frac{a^{*}\big(\int_{\mathbb{R}}|u_{1}|^{4}dx+\int_{\mathbb{R}}|u_{2}|^{4}dx\big)}{a_{1}\int_{\mathbb{R}}|u_{1}|^{4}dx+a_{2}\int_{\mathbb{R}}|u_{2}|^{4}dx+2\beta_{*}\int_{\mathbb{R}}|u_{1}|^{2}|u_{2}|^{2}dx} \notag\\
			&\label{eq2.16}\geq\frac{a^{*}\big(\int_{\mathbb{R}}|u_{1}|^{4}dx+\int_{\mathbb{R}}|u_{2}|^{4}dx\big)}{a_{1}\int_{\mathbb{R}}|u_{1}|^{4}dx+a_{2}\int_{\mathbb{R}}|u_{2}|^{4}dx+2\beta_{*} \big(\int_{\mathbb{R}}|u_{1}|^{4}dx\int_{\mathbb{R}}|u_{2}|^{4}dx\big)^{\frac{1}{2}}} \\
			&=\kappa _{a_{1}, a_{2}, \beta_{*}}(\omega_{0})\ \ \text{with} \ \ \omega_{0}:=\Big(\frac{\int_{\mathbb{R}}|u_{2}|^{4}dx}{\int_{\mathbb{R}}|u_{1}|^{4}dx}\Big)^{\frac{1}{2}}>0\notag \\
			&\label{eq2.17}\geq\mathop{\mathrm{inf\mathrm{} } }\limits _{t\in(0,\infty) }\kappa _{a_{1}, a_{2}, \beta_{*}}(t) ,
		\end{align}
		\noindent where
		\begin{equation}\label{eq2.18}
			\kappa_{a_1, a_2, \beta_*}(t):=\frac{a^*(1+t^2)}{a_1+a_2t^2+2\beta_*t}.
		\end{equation}
		
		\noindent We know from the H\"{o}lder  inequality that the inequality (\ref{eq2.16}) becomes equality
		 if and only if
		\begin{equation}\label{eq2.19}
			u_2^2(x)=l u_1^2(x)\ \ \text{for some}\ \ l>0.
		\end{equation}
		\noindent  From the definition (\ref{eq2.18}), it is easy to verify that
		\begin{equation}\label{eq2.20}
			\kappa_{a_1, a_2,  \beta_{*}}(t)\geq1,\ \ \forall\:t\in(0,+\infty),
		\end{equation}
		and
		\begin{equation}
			\kappa_{a_1, a_2, \beta_*}(t)=1\:\Leftrightarrow \:t=t_0:=\sqrt{\frac{a^*-a_1}{a^*-a_2}} .
		\end{equation}
		
		\noindent  Combining (\ref{eq2.17}) and (\ref{eq2.20}) gives that
		\begin{equation*}
			\Gamma\big(a_1,a_2,\beta_*\big)\geq1,
		\end{equation*}
		which means that proving (i) requires us only to exclude the possibility that 
		
		\noindent $\Gamma\big(a_1,a_2,\beta_*\big)=1$. In fact, if
		\begin{equation}\label{eq2.22}
			\Gamma\big(a_1,a_2,\beta_*\big)=1,
		\end{equation}
		then (\ref{eq2.19}) holds, i.e.,
		\begin{equation}\label{eq2.23}
			u_2^2(x)=lu_1^2(x), \ \ \text{and}  \ \ l=t_{0}=\sqrt{\frac{a^*-a_1}{a^*-a_2}}>1.
		\end{equation}
		This and (\ref{eq2.13})  imply that
		\begin{align}\label{eq2.24}
			\Gamma(a_1,a_2,\beta_*)\notag& =\lim_{n\to\infty}J_{a_{1},a_{2},\beta_{*}}\big(u_{1n},u_{2n}\big)\notag \\
			&\geq\frac{2\|(-\Delta)^{\frac14}u_1\|_2^2+2\|(-\Delta)^{\frac14}u_2\|_2^2}{a_{1}\int_{\mathbb{R}}|u_{1}|^{4}dx+a_{2}\int_{\mathbb{R}}|u_{2}|^{4}dx+2\beta_{*}\int_{\mathbb{R}}|u_{1}|^{2}|u_{2}|^{2}dx} \\
			&=\frac{2(1+l)}{a_1+a_2l^2+2\beta_*l}\cdot\frac{\|(-\Delta)^{\frac14}u_1\|_2^2}{\int_{\mathbb{R}}|u_{1}|^{4}dx}. \notag
		\end{align}
		On the other hand, let
		\begin{equation}\label{eq2.25}
			\tilde{u}_i=\frac{1}{\sqrt{\lambda_i}}u_i\ \ \text{ with }\ \ \lambda_i:=\int_{\mathbb{R}}\left|u_i\right|^2dx\leq\lim_{n\to\infty}\int_{\mathbb{R}}\left|u_{in}\right|^2dx=1.
		\end{equation}
		\noindent Then, we have
		\begin{equation*}
			\int_{\mathbb{R}}|\tilde{u}_i|^2dx=1,\ \  i=1,2.
		\end{equation*}
		By (\ref{eq2.23}), we further have
		\begin{equation}\label{eq2.26}
			\lambda_{2}=l\lambda_{1}\leq1.
		\end{equation}
		It then follows from (\ref{eq2.23}), (\ref{eq2.25}) and (\ref{eq2.26}) that
		\begin{align*}
			\Gamma(a_1,a_2,\beta_*) & \leq\frac{2\|(-\Delta)^{\frac{1}{4}}\tilde{u}_{1}\|_{2}^{2}+2\|(-\Delta)^{\frac{1}{4}}\tilde{u}_{2}\|_{2}^{2}}{a_{1}\int_{\mathbb{R}}|\tilde{u}_1|^{4}dx+a_{2}\int_{\mathbb{R}}|\tilde{u}_2|^{4}dx+2\beta_{*}\int_{\mathbb{R}}|\tilde{u}_1|^{2}|\tilde{u}_2|^{2}dx} \\
			&=\frac{4\lambda_1}{a_1+a_2+2\beta_*}\cdot\frac{\|(-\Delta)^{\frac14}u_1\|_2^2}{\int_{\mathbb{R}}|u_{1}|^{4}dx}\leq\frac{4/l}{a_1+a_2+2\beta_*}\cdot\frac{\|(-\Delta)^{\frac14}u_1\|_2^2}{\int_{\mathbb{R}}|u_{1}|^{4}dx}.
		\end{align*}
		Putting together this with (\ref{eq2.24}), we infer that
		\begin{equation*}
			\frac{2(1+l)}{a_1+a_2l^2+2\beta_*l}\leq\frac{4/l}{a_1+a_2+2\beta_*},
		\end{equation*}
		i.e.,
		\begin{equation}\label{eq2.27}
			\frac{2(l+l^2)}{a_1+a_2l^2+2\beta_*l}\leq\frac4{a_1+a_2+2\beta_*},\ \ \mathrm{where}\ \ l=\sqrt{\frac{a^*-a_1}{a^*-a_2}}>1.
		\end{equation}
		Next, we show  that  (\ref{eq2.27}) cannot be true, which implies that assumption  (\ref{eq2.22}) is false, thereby proving (i).
		
		Since $0<a_1<a_2<a^*$, it is easy to see that
		\begin{equation*}
			(a^*-a_1)-(a^*-a_2)\leq2\sqrt{(a^*-a_1)(a^*-a_2)}.
		\end{equation*}
		We further have
		\begin{equation*}
			l^2-1\le 2l,
		\end{equation*}
		which implies that   $1<l\leq1+\sqrt{2}$. Let
		\begin{equation*}
			z(t):=\frac{2(t+t^2)}{a_1+a_2t^2+2\beta_*t}.
		\end{equation*}
		It is easy to see  that $z(t)$ is strictly increasing function as $t\in [1, 1+\sqrt{2} ]$, because
		\begin{equation*}
			z'(t)=2\frac{(2\beta_*-a_2)t^2+2a_1t+a_1}{\left(a_1+a_2t^2+2\beta_*t\right)^2},
		\end{equation*}
		and by (\ref{eq2.15}), for $t\in[1,1+\sqrt{2})$,
		\begin{equation*}
			(2\beta_*-a_2)t^2+2a_1t+a_1\geq-a_1t^2+2a_1t+a_1=a_1[2-(t-1)^2]>0,
		\end{equation*}
		which shows that $z'(t)>0$. Therefore, $z(1)<z(l)$ by $l\in (1,1+\sqrt{2}]$, and then (\ref{eq2.27}) cannot be hold.
		
		(\textbf{ii}): It follows from  (\ref{eq2.4}) that,  for any $(a_{1},a_{2},\beta ) \in \mathbb{{R}}^{3}_{+}\setminus \left \{ (0,0,0)\right \}$,
		\begin{equation*}
			\Gamma(a_1,a_2,\beta)\leq\frac{2a^*}{a_1+a_2+2\beta}.
		\end{equation*}
		Since $0<a_1<a_2<a^*$, and taking $\beta=\beta^*=\frac{2a^*-a_1-a_2}2$ into the above inequality results in   $\Gamma(a_1,a_2,\beta^*)\leq1$, this completes the proof of (ii).
		
	\end{proof}

	Finally, in order to obtain the existence and non-existence criteria of minimizer, that is, to prove Theorem \ref{theorem 2.4}. We need to establish some estimates for $\psi_R$, i.e.,
	\begin{align}
		\label{eq2.48}\int_{\mathbb{R}}\psi_R\sqrt{-\Delta}\psi_R dx\leq\frac{R}{\|Q\|_{2}^{2}}\|(-\Delta)^{\frac{1}{4}}Q\|_{2}^{2}+O(R^{-\frac{3}{2}})=R+O(R^{-\frac{3}{2}}).
			\end{align}
\begin{equation}
	\begin{split}
				\frac{2R}{a^{*}}-O(R^{-6})& =\frac{R}{||Q||_{2}^{4}}\int_{\mathbb{R}}Q^{4}dx-O(R^{-6})\leq\int_{\mathbb{R}}|\psi_R|^{4}dx \\
		&\label{eq2.49}\leq\frac{R}{\|Q\|_{2}^{4}}\int_{\mathbb{R}}Q^{4}dx+O(R^{-2})=\frac{2R}{a^{*}}+O(R^{-2}).
	\end{split}
\end{equation}

	Here, $\psi_R$  is defined as
	\begin{equation}\label{eq2.33}
		\psi_R=A_R\frac{R^{\frac{1}{2}}}{\|Q\|_2}\varphi(x-x_0)Q\big(R(x-x_0)\big),
	\end{equation}
	where $x_0\in\mathbb{R}$, $R>1$, $Q$ is the unique  positive  radially symmetric solution of (\ref{classical eq}), $\varphi(x)\in C_0^\infty(\mathbb{R})$ is a non-negative function, which satisfies
	\begin{equation}\label{fai}
		\varphi(x)=1\ \ \text{for}\ \ |x|\leq\frac{1}{2}; \ \ \varphi(x)=0 \ \ \text{for} \ \ |x|>1; \ \ 0\leq\varphi\leq1,
	\end{equation}
	and $A_{R}>0$ is chosen such that $\int_{\mathbb{R}}\psi_R^2dx=1$, that is,
	\begin{equation}\label{eq 1}
		1=\int_{\mathbb{R}}\frac{A_R^2R}{\|Q\|_2^2}\varphi^2(x)Q^2(Rx) dx=\int_{\mathbb{R}}\frac{A_R^2}{\|Q\|_2^2}\varphi^2(R^{-1}x)Q^2(x) dx.
	\end{equation}
	
	We now  prove (\ref{eq2.48}) and (\ref{eq2.49}),   it follows from the decay rate of $Q$ given in (\ref{decay}) and (\ref{eq 1}) that, for $R$ large enough,
	\begin{align}
		\begin  {vmatrix}1-A_R^2\end{vmatrix}&=\left|\int_{\mathbb{R}}\frac{A_R^2}{\|Q\|_2^2}\Big(\varphi^2\big(\frac{x}{R}\big)-1\Big)Q^2(x) dx\right| \notag \\
	&\leq\int_{\mathbb{R}\setminus B_{\frac{1}{2}R}}\left|\frac{A_{R}^{2}}{\|Q\|_{2}^{2}}\Big(\varphi^{2}\big(\frac{x}{R}\big)-1\Big)Q^{2}(x)\right| dx \\
	&\leq C\int_{\mathbb{R}\setminus B_{\frac{1}{2}R}}|x|^{-4} dx\notag \\
	&\leq CR^{-3}.\notag
\end{align}
Then we  have
\begin{equation}\label{AR}
	\begin{aligned}&1\leq A_{R}^{2}=1-\int_{\mathbb{R}}\frac{A_{R}^{2}}{\|Q\|_{2}^{2}}\Big(\varphi^{2}\big(\frac{x}{R}\big)-1\Big)Q^{2}(x) dx\leq1+O(R^{-3})\end{aligned}	
\end{equation}
as $R\to\infty$.

Now, we focus on the translations and scaling of integrals involving the non-local operator $\sqrt{-\Delta }$ and show that
\begin{align}\label{eq2.37}
	&\frac{\|Q\|_{2}^{2}}{A_{R}^{2}}\int_{\mathbb{R}}\psi_R\sqrt{-\Delta}\psi_R dx=R\int_{\mathbb{R}}\varphi(R^{-1}x)Q(x)\sqrt{-\Delta}\big(\varphi(R^{-1}x)Q(x)\big) dx.\end{align}

We  first address the translation. By the definition of $\sqrt{-\Delta}$ given in (\ref{four}), we have
\begin{align}\label{eq2.38}
	&\frac{\|Q\|_2^2}{A_R^2}\int_{\mathbb{R}^3}\psi_R\sqrt{-\Delta}\psi_R dx \notag\\
	&=\int_{\mathbb{R}}R\varphi(x-x_{0})Q(R(x-x_{0}))\mathcal{F}^{-1}\Big(|\xi|\mathcal{F}\big(\varphi(y-x_{0})Q(R(y-x_{0}))\big)\Big) dx \notag\\
	&=\frac{1}{2\pi}\int_{\mathbb{R}}R\varphi(x)Q(Rx) dx\int_{\mathbb{R}}\mathrm{e}^{i(x+x_{0})\cdot\xi}|\xi| d\xi\int_{\mathbb{R}}\mathrm{e}^{-i\xi\cdot(y+x_{0})}\varphi(y)Q(Ry) dy \\
	&=\int_{\mathbb{R}}R\varphi(x)Q(Rx)\mathcal{F}^{-1}\big(|\xi|\mathcal{F}(\varphi(y)Q(Ry))\big) dx \notag\\
	&=\int_{\mathbb{R}}R\varphi(x)Q(Rx)\sqrt{-\Delta}\big(\varphi(x)Q(Rx)\big) dx. \notag
\end{align}\\
On the other hand, regarding the scaling of integrals, by applying the Plancherel theorem and (\ref{exch}), we  deduce that
\begin{align}\label{eq2.39}
	&\int_{\mathbb{R}}R\varphi(x)Q(Rx)\sqrt{-\Delta}\big(\varphi(x)Q(Rx)\big) dx \notag\\
	& =\int_{\mathbb{R}}R|\xi|\mathcal{F}\big(\varphi(y)Q(Ry)\big)^{2} d\xi \notag \\
	&= \begin{aligned}\frac{R}{2\pi}\int_{\mathbb{R}}|\xi|d\xi\Big(\int_{\mathbb{R}}\mathrm{e}^{-i\xi\cdot y}\varphi(y)Q(Ry) dy\Big)^{2}\end{aligned}\notag \\
	&= \frac{R}{2\pi}\int_{\mathbb{R}}|\xi|d\xi\Big(\int_{\mathbb{R}}\mathrm{e}^{-i\xi\cdot y}\varphi(R^{-1}y)Q(y) dy\Big)^{2} \\
	&=R\int_{\mathbb{R}}|\xi|\mathcal{F}\big(\varphi(R^{-1}y)Q(y)\big)^{2} d\xi  \notag\\
	&=R\int_{\mathbb{R}}\varphi(R^{-1}x)Q(x)\sqrt{-\Delta}\big(\varphi(R^{-1}x)Q(x)\big) dx. \notag
\end{align}
It then follows from the  (\ref{eq2.38}) and (\ref{eq2.39}) that   (\ref{eq2.37}) holds.
Equation (\ref{eq2.37}) also implies that
\begin{align}\label{eq2.40}
	&\frac{\|Q\|_2^2}{A_R^2}\int_{\mathbb{R}}\psi_R\sqrt{-\Delta}\psi_R dx \notag\\
	&=R\int_{\mathbb{R}}\varphi(R^{-1}x)Q(x)\sqrt{-\Delta}\big(\varphi(R^{-1}x)Q(x)\big) dx\notag \\
	&=R\bigg(\int_{\mathbb{R}}Q(x)\sqrt{-\Delta}Q(x) dx+\int_{\mathbb{R}}\big(\varphi(R^{-1}x)-1\big)Q(x)\sqrt{-\Delta}Q(x) dx \\
	&\quad\quad+\int_{\mathbb{R}}\varphi(R^{-1}x)Q(x)\sqrt{-\Delta}\Big(\big(\varphi(R^{-1}x)-1\big)Q(x)\Big) dx\bigg) \notag\\
	&=:R\int_{\mathbb{R}}Q(x)\sqrt{-\Delta}Q(x) dx+RA+RB, \notag
\end{align}
$A:=\int_{\mathbb{R}}\big(\varphi(R^{-1}x)-1\big)Q(x)\sqrt{-\Delta}Q(x) dx$, $B:=\int_{\mathbb{R}}\varphi(R^{-1}x)Q(x)\sqrt{-\Delta}\Big(\big(\varphi(R^{-1}x)-1\big)Q(x)\Big) dx$.

We now turn to estimate terms $A$ and $B$. Since $Q(x)$ is a positive  solution of (\ref{classical eq}), it is easy to see that
\begin{equation}\label{eq2.41}
	\big | \sqrt{-\Delta }Q(x)  \big |\ \leq  Q^3(x)  +Q(x).
\end{equation}
Combining the decay rate of $Q$  (\ref{decay}) and (\ref{eq2.41})  gives that
\begin{align}\label{eq2.42}
	\left|A\right|& \leq\int_{\mathbb{R}\setminus B_{\frac12R}}\left|\big(\varphi(R^{-1}x)-1\big)Q(x)\sqrt{-\Delta}Q(x)\right|dx \notag\\
	&\leq C\int_{\mathbb{R}\setminus B_{\frac12R}}|x|^{-8}+|x|^{-4}dx \\
	&\leq C(R^{-7}+R^{-3}). \notag
\end{align}

In order to estimate the term $B$, we first recall an estimate of commutators. Let $\varphi_R:=\varphi(R^{-1}x)$.  It is known from Remark 4 in \cite{NEW1} that
\begin{equation}\label{eq2.43}
	\|[\sqrt{-\Delta},\varphi_R]\|_{L^2(\mathbb{R})}\leq C\|\nabla\varphi_R\|_\infty.
\end{equation}
Here $[\sqrt{-\Delta},\varphi_R]$ denotes  the commutator of $\sqrt{-\Delta}$ and $\varphi_R$, and $[\cdot,\cdot]$ represents the Lie bracket. By (\ref{eq2.43}), we  derive that
\begin{equation}\label{eq2.44}
	\int_{\mathbb{R}}\Big([\sqrt{-\Delta},\varphi_R]Q\Big)^2dx\leq C\|\nabla\varphi_R\|_\infty^2\|Q\|_2^2 .
\end{equation}
Obviously,
\begin{align*}
	\left|B\right|& \leq\Big|\int_{\mathbb{R}}\big(\varphi(R^{-1}x)-1\big)Q(x)\varphi(R^{-1}x)\sqrt{-\Delta}Q(x) dx\Big| \\
	&\quad\quad+\Big|\int_{\mathbb{R}}\big(\varphi(R^{-1}x)-1\big)Q(x)[\sqrt{-\Delta},\varphi_R]Q(x) dx\Big| \\
	&\leq C \int_{\mathbb{R}\setminus B_{\frac12R}}\Big(|Q(x)||\sqrt{-\Delta}Q(x)|+|Q(x)|\big|[\sqrt{-\Delta},\varphi_{R}]Q(x)\big|\Big)dx \\
	&=B_{1}+B_{2}.
\end{align*}
According to the decay rate of $Q$ (\ref{decay}) and (\ref{eq2.41}), we obtain that
\begin{equation*}
	|B_1|\leq C\int_{\mathbb{R}\setminus B_{\frac12R}}|x|^{-8} +|x|^{-4}dx\leq C(R^{-7}+R^{-3}).
\end{equation*}
It  follows from the H\"{o}lder inequality and (\ref{eq2.44}) that
\begin{align*}
	\left|B_{2}\right|& \leq C\Big(\int_{\mathbb{R}\setminus B_{\frac12R}}Q^{2}(x) dx\Big)^{\frac{1}{2}}\|[\sqrt{-\Delta},\varphi_{R}]Q\|_{2} \\
	&\leq C\Big(\int_{\mathbb{R}\setminus B_{\frac12R}}|x|^{-4}dx\Big)^{\frac{1}{2}}\|\nabla\varphi_R\|_\infty\|Q\|_2 \\
	&\leq CR^{-\frac32}R^{-1}\|Q\|_{2}^{2} \\
	&\leq CR^{-\frac52}.
\end{align*}
As a consequence,
\begin{equation}\label{eq2.45}
	\begin{split}
		&\int_{\mathbb{R}}\psi_R\sqrt{-\Delta}\psi_R dx \\
		&\leq\frac R{\|Q\|_2^2}\int_{\mathbb{R}}Q(x)\sqrt{-\Delta}Q(x) dx+O(R^{-\frac32}).
	\end{split}
\end{equation}
Next, we consider the following term,
\begin{align*}
	\int_{\mathbb{R}}|\psi_R|^4 dx
	&=\frac{A_{R}^{4}R^{2}}{\|Q\|_{2}^{4}}\int_{\mathbb{R}}\varphi^{4}(x-x_{0})Q^{4}(R(x-x_{0}))dx \\
	&=\frac{A_{R}^4R}{\|Q\|_{2}^{4}}\int_{\mathbb{R}}\varphi^{4}(R^{-1}x)Q^{4}(x)dx \\
	&=:\frac{A_R^4R}{\|Q\|_2^4}F,
\end{align*}
where $F$  can be expressed as
\begin{align*}
	\text{F}& =\int_{\mathbb{R}}\varphi^{4}(R^{-1}x)Q^{4}(x) dx \\
	&=\int_{\mathbb{R}}Q^{4}(x)dx+\int_{\mathbb{R}}\big(\varphi^{4}(R^{-1}x)-1\big)Q^{4}(x)dx \\
	&=:\int_{\mathbb{R}}Q^{4}(x)dx+F_{1}.
\end{align*}
By the decay rate of $Q$ (\ref{decay}), we have
\begin{align*}|F_1|& \leq\int_{\mathbb{R}}\Big|\varphi^{4}(R^{-1}x)-1\Big|Q^{4}(x)dx \\
	&\leq\int_{\mathbb{R}\setminus B\frac{1}{2}R}Q^{4}(x)dx \\
	&\leq CR^{-7}.
\end{align*}
Therefore, according to (\ref{AR}), it is deduced that
\begin{align}\label{eq2.46}
	\int_{\mathbb{{R}}}|\psi_R|^{4}dx
	&=\frac{A_R^{4}R}{||Q||_{2}^{4}}\Big(\int_{\mathbb{R}}Q^{4}(x)dx+F_{1}\Big)\notag \\
	&\geq\frac{R}{\|Q\|_{2}^{4}}\Big(\int_{\mathbb{R}}Q^{4}(x)dx+F_{1}\Big) \notag\\
	&\geq\frac{R}{\|Q\|_{2}^{4}}\Big(\int_{\mathbb{R}}Q^{4}(x)dx-|F_{1}|\Big) \\
	&\geq\frac{R}{\|Q\|_{2}^{4}}\Big(\int_{\mathbb{R}}Q^{4}(x)dx-O\big(R^{-7})\Big) \notag\\
	&=\frac{2R}{a^{*}}-O\big(R^{-6}\big). \notag
\end{align}
Moreover, by (\ref{AR}), we also have
\begin{align}\label{eq2.47}
	\int_{\mathbb{R}}|\psi_R|^{4}dx
	&=\frac{A_{R}^{4}R}{||Q||_{2}^{4}}\Big(\int_{\mathbb{R}}Q^{4}(x)dx+F_{1}\Big) \notag\\
	&\leq\frac{R}{\|Q\|_{2}^{4}}\Big(\int_{\mathbb{R}}Q^{4}(x)+|F_{1}|\Big)\Big(1+O(R^{-3})\Big)^{2} \notag\\
	&\le\frac{R}{||Q||_{2}^{4}}\Big(\int_{\mathbb{{R}}}Q^{4}(x)dx+Q\big(R^{-3}\big)\Big) \\
	&=\frac{R}{||Q||_{2}^{4}}\int_{\mathbb{{R}}}Q^{4}(x)+O\big(R^{-2}\big)\notag\\
	&=\frac{2R}{a^{*}}+O\big(R^{-2}\big). \notag
\end{align}
Consequently, (\ref{eq2.48})  can be derived from  (\ref{eq1.11}) and (\ref{eq2.45}). Furthermore, by combining (\ref{eq2.46}) and (\ref{eq2.47}), we obtain (\ref{eq2.49}).

\begin{theorem}\label{theorem 2.4}
	Suppose that condition $(\mathcal{D}_{1})$ holds, and let $a_{1},a_2 >0$,\ \ $\beta\in \mathbb{R}$, then we have:\\
		\noindent \textnormal{(i)} If $\Gamma(a_1,a_2,\beta)>1$, there exists a minimizer for  problem (\ref{problem 1.7}). \\	
	\noindent \textnormal{(ii)} If $\Gamma(a_1,a_2,\beta)<1$, there does not exist any minimizer for problem (\ref{problem 1.7}).
\end{theorem}

\begin{proof}
	(i)	Let $\left \{ (u_{1n},u_{2n}) \right \}  \subset \mathcal{M}$ be a minimizing sequence of problem (\ref{problem 1.7}), then
	\begin{equation*}
		\|u_{1n}\|_2^2=\|u_{2n}\|_2^2=1\quad\mathrm{and}\quad\lim_{n\to\infty}E_{a_1,a_2,\beta}(u_{1n},u_{2n})=\hat{e}(a_1,a_2,\ \beta).
	\end{equation*}
	It follows from the definition of $\Gamma(a_1,a_2,\beta)$  in (\ref{problem2.1}) that
	\begin{equation}\label{eq2.28}
\begin{split}
			E_{a_1,a_2,\beta}(u_{1n},u_{2n})	&\geq(1-\frac{1}{\Gamma(a_{1},a_{2},\beta)})\Big[\|(-\Delta)^{\frac{1}{4}}u_{1n}\|_{2}^{2}+\|(-\Delta)^{\frac{1}{4}}u_{2n}\|_{2}^{2}\Big]\\
	&\quad\quad+\sum_{i=1}^{2}\int_{\mathbb{R} }  V_{i}(x)|u_{in}|^{2}dx.
\end{split}
	\end{equation}
	If $\Gamma(a_1,a_2,\ \beta)>1$, we see from (\ref{eq2.28}) that $\{(u_{1n},u_{2n})\}$ is bounded in $\mathcal{X}$.  Therefore, by the compactness of  Lemma \ref{lemma2.1}, there exists $(u_1,u_2)\in\mathcal{X}$ such that
	\begin{align*}&(u_{1n},u_{2n})\stackrel{n}{\rightharpoonup}(u_{1},u_{2})\ \text{ weakly in }\mathcal{X},\\&(u_{1n},u_{2n})\stackrel{n}{\to}(u_{1},u_{2})\ \text{ strongly in }L^{q}(\mathbb{R})\times L^{q}(\mathbb{R})\text{ for }q\in[2,\infty).\end{align*}
	Moreover, it is known from Lemma \ref{lemma limit} that
	\begin{equation*}
		\lim_{n\to\infty}\int_{\mathbb{R}}|u_{in}|^{4}dx=\int_{\mathbb{R}}|u_{i}|^{4}dx,\ \ i=1, 2.
	\end{equation*}
	Since  $\left(\sqrt{-\Delta}u,u\right)$ is weakly lower semi-continuous and $\|u_1\|_2^2=\|u_2\|_2^2=1$, we   yield  that
	\begin{equation*}
		\hat{e}(a_1,a_2,\beta)\leq E_{a_1,a_2, \beta}(u_1,u_2)\leq\lim_{n\to\infty}E_{a_1,a_2, \beta}(u_{1n},u_{2n})=\hat{e}(a_1,a_2, \beta),
	\end{equation*}
	which implies that $E_{\alpha_1,\alpha_2, \beta}(u_1,u_2)=\hat{e}(a_1,a_2,\beta)$, i.e., $(u_1,u_2)$ is a minimizer of problem (\ref{problem 1.7}).  This completes the proof of part (i).
	
	(ii) Suppose now that $\Gamma(a_1,a_2,\beta)<1$, then one may choose $(u_1,u_2)\in \mathcal{M}$ such that $u_1$ and $u_2$ have compact support in $\mathbb{R}$ and satisfies
	\begin{equation}\label{eq2.29}
\begin{split}
		&\frac{2\|(-\Delta)^{\frac14}u_1\|_2^2+2\|(-\Delta)^{\frac14}u_2\|_2^2}{a_{1}\int_{\mathbb{R}}|u_{1}|^{4}dx+a_{2}\int_{\mathbb{R}}|u_{2}|^{4}dx+2\beta^{+}\int_{\mathbb{R}}|u_{1}|^{2}|u_{2}|^{2}dx}\\ &\leq\rho_0:=\frac{1+\Gamma(a_1,a_2,\beta)}2<1.
\end{split}
	\end{equation}
	For $\lambda>0$, denote
	\begin{equation}\label{eq2.30}
		\hat{u}_i(x)=\lambda^{\frac{1}{2}}u_i(\lambda x),\ \ i=1, 2.
	\end{equation}
	It is easy to see that $(\hat{u}_1,\hat{u}_2)\in\mathcal{M}$.
	Since $u_i(x)$ has compact support in $\mathbb{{R}}$ and $V_i(x)\in L_{\mathrm{loc}}^\infty(\mathbb{R})$, there exists  a positive constant $C$, independent of $\lambda>0$,  such that for $\lambda\to \infty$,
	\begin{equation}\label{eq2.31}
		\int_{\mathbb{R}}V_i(x)|\hat{u}_i|^2dx=\int_{\mathbb{R}}V_i(\frac{x}{\lambda})|u_i|^2dx\leq C<\infty, \ \ i=1,2,
	\end{equation}
	which, together with  (\ref{eq2.29}) and (\ref{eq2.30}),  implies that, for $\beta\ge0$,
	\begin{align}\label{eq2.32}
		\begin{aligned}E_{a_{1},a_{2},\beta}(\hat{u}_{1},\hat{u}_{2})& =\sum_{i=1}^2\Big(\|(-\Delta)^{\frac14}\hat{u}_i\|_2^2+\int_{\mathbb{R}}V_i(x)|\hat{u}_i|^2dx-\frac{a_i}2\int_{\mathbb{R}}|\hat{u}_i|^{4}dx\Big) \\&\quad-\beta\int_{\mathbb{R}}|\hat{u}_1|^{2}|\hat{u}_2|^{2}dx \\&=\lambda\left[\sum_{i=1}^2\Big(\|(-\Delta)^{\frac14}u_i\|_2^2-\frac{a_i}2\int_{\mathbb{R}}|u_i|^{4}dx\Big)-\beta\int_{\mathbb{R}}|u_1|^{2}|u_2|^{2}dx\right] \\&\quad+\sum_{i=1}^2\int_{\mathbb{R}}V_i(x)|\hat{u}_i|^2dx\\&\leq\lambda(\rho_0-1)\Big(\sum_{i=1}^2\frac{a_i}2\int_{\mathbb{R}}|u_{i}|^{4}dx+\beta\int_{\mathbb{R}}|u_{1}|^{2}|u_{2}|^{2}dx\Big)\\&\quad+\sum_{i=1}^2\int_{\mathbb{R}}V_i(x)|\hat{u}_i|^2dx\\ &\to-\infty\ \ \mathrm{as}\ \  \lambda\to+\infty.\end{aligned}
	\end{align}
	This implies  that
	\begin{equation*}
		\hat{e}(a_1,a_2,\beta)\leq E_{a_1,a_2,\beta}(\hat{u}_1,\hat{u}_2)\to-\infty\ \ \mathrm{as}\ \ \lambda\to+\infty,
	\end{equation*}
	which means  that $\hat{e}(a_1,a_2,\beta)$ can not  admit any  minimizer if $\beta\geq0$.
	
	Now, we  consider the case  $\beta<0$, i.e., $\beta^+=0$. If  $\Gamma(a_1,a_2,\beta)<1$, we have
	\begin{equation*}
		a_1>a^*\ \ \text{or }\ \ a_2>a^*.
	\end{equation*}
	Indeed, if  $0<a_1,a_2\le a^*$,  according to the Gagliardo-Nirenberg inequality (\ref{eq1.10}), for $u_i\in H^{\frac12}(\mathbb{R})$ satisfying $\|u_i\|_2^2=1$,  $i=1,2$, we have
	\begin{equation*}
		\begin{aligned} J_{a_1 ,a_2,\beta}(u_1,u_2)&=\frac{2\|(-\Delta)^{\frac14}u_1\|_2^2+2\|(-\Delta)^{\frac14}u_2\|_2^2}{a_{1}\int_{\mathbb{R}}|u_{1}|^{4}dx+a_{2}\int_{\mathbb{R}}|u_{2}|^{4}dx}\geq\frac{a^*\int_{\mathbb{R}}|u_{1}|^{4}dx+a^*\int_{\mathbb{R}}|u_{2}|^{4}dx}{a_1\int_{\mathbb{R}}|u_{1}|^{4}dx+a_2\int_{\mathbb{R}}|u_{2}|^{4}dx}\\&\geq1.\end{aligned}
	\end{equation*}
	This demonstrates that $\Gamma(a_1,a_2,\beta)\ge 1$, which is in contradiction with our assumption. Therefore, without loss of generality, we now suppose that $a_1>a^*$.
	
	Since the bounded function $x\mapsto V_{i}(x)\varphi^{2}(\frac{x-x_{0}}{R})$  has compact support, where $\varphi$ is given by (\ref{fai}). Using the dominated convergence theorem, we obtain the following result, for $\psi_R$ defined in (\ref{eq2.33}),
	\begin{align}\label{eq2.50}
		\operatorname*{lim}_{R\to\infty}\int_{\mathbb{R}}V_{i}(x)\psi_R^{2}dx& =\lim_{R\to\infty}\int_{\mathbb{R}}V_i(x)A_R^2\frac{R}{\|Q\|_2^2}\varphi^2(x-x_0)Q^2\big(R(x-x_0)\big)dx \notag\\
		&=\lim_{R\to\infty}\frac{A_R^2}{\|Q\|_2^2}\int_{\mathbb{R}}V_i(R^{-1}x+x_0)\varphi^2(R^{-1}x)Q^2(x)dx \\
		&=V_i(x_0)\ \ i=1,2.\notag
	\end{align}
	Moreover, based on the H\"{o}lder inequality and  (\ref{eq2.49}), it can be inferred that
	\begin{equation}\label{eq2.51}
\begin{split}
			\int_{\mathbb{R}}|\psi_R|^2|\nu|^2dx&\leq\Big(\int_{\mathbb{R}}|\psi_R|^4dx\int_{\mathbb{R}}|\nu|^4dx\Big)^{\frac12}\\
	&=\sqrt{\frac2{a^*}}R^{\frac12}\Big(\int_{\mathbb{R}}|\nu|^4dx\Big)^{\frac12}+O(R^{-1}),
\end{split}
	\end{equation}
	which, together with (\ref{eq2.48}), (\ref{eq2.49}) and (\ref{eq2.50}), implies that
	for $0\leq\nu\in C_0^\infty(\mathbb{R})$ with $\int_{\mathbb{R}}|\nu|^2dx=1$,
	\begin{align*}
		E_{a_{1},a_{2},\beta}(\psi_R,\nu)& =\|(-\Delta)^{\frac{1}{4}}\psi_R\|_{2}^{2}+\|(-\Delta)^{\frac{1}{4}}\nu\|_{2}^{2}-\frac{a_{1}}{2}\int_{\mathbb{R}}|\psi_R|^{4}dx-\frac{a_{2}}{2}\int_{\mathbb{R}}|\nu|^{4}dx \\
		&\quad-\beta\int_{\mathbb{R}}|\psi_R|^{2}|\nu|^{2}dx+\int_{\mathbb{R}}V_1(x)\psi_R^{2}dx+\int_{\mathbb{R}}V_2(x)\nu^2dx \\
		&\leq R\big(\!1-\frac{a_1}{a^*}\!\big)\!+ \! \|(-\Delta)^{\frac14}\nu\|_2^2\!- \!\frac{a_2}2\int_{\mathbb{R}}|\nu|^{4}dx \!+ \! |\beta|\sqrt{\frac2{a^*}}R^{\frac12}\Big(\!\int_{\mathbb{R}}|\nu|^{4}dx\!\Big)^{\frac12} \\
		&\quad+\int_{\mathbb{R}}V_1(x)\psi_R^2dx+\int_{\mathbb{R}}V_2(x)\nu^2dx+O(R^{-1})+O(R^{-\frac32})+O(R^{-6}) \\
		&\to-\infty\ \ \text{ as }\ \ R\to+\infty.
	\end{align*}
	\renewcommand{\qedsymbol}{}
	The case   $a_2>a^*$ can be addressed in a similar manner, therefore,  $	\hat{e} (a_{1}, a_{2},  \beta)$ has no minimizer for $\beta<0$, either.
\end{proof}
\renewcommand{\qedsymbol}{}
\end{proof}
\section{Proofs of the main theorems}\label{section3}
In this section, we prove Theorems \ref{theorem1.1}, \ref{theorem1.2}, and \ref{theorem1.3}.

\textbf{Proof of Theorem} \ref{theorem1.1}. \textbf{(i):} For any $(u_{1},u_2)\in \mathcal{M}$, applying the  H\"older inequality and the Gagliardo-Nirenberg inequality  (\ref{eq1.10}) yields that
\begin{align*}
	J_{a_{1},a_{2},\beta}(u_{1},u_{2})&\geq\frac{a^{*}\Big(\int_{\mathbb{R}}|u_{1}|^{4}dx+\int_{\mathbb{R}}|u_{2}|^{4}dx\Big)}{a_{1}\int_{\mathbb{R}}|u_{1}|^{4}dx+a_{2}\int_{\mathbb{R}}|u_{2}|^{4}dx+2\beta^{+}\Big(\int_{\mathbb{R}}|u_{1}|^{4}dx\int_{\mathbb{R}}|u_{2}|^{4}dx\Big)^{\frac{1}{2}}}\\&=\frac{a^{*}\Big(1+\int_{\mathbb{R}}|u_{2}|^{4}dx/\int_{\mathbb{R}}|u_{1}|^{4}dx\Big)}{a_{1}+a_{2}\int_{\mathbb{R}}|u_{2}|^{4}dx/\int_{\mathbb{R}}|u_{1}|^{4}dx+2\beta^{+}\Big(\int_{\mathbb{R}}|u_{2}|^{4}dx/\int_{\mathbb{R}}|u_{1}|^{4}dx\Big)^{\frac{1}{2}}}.
\end{align*}
Let
\begin{equation}\label{eq3.1}
	\kappa_{a_{1}, a_{2}, \beta^+}(t):=\frac{a^*(1+t^2)}{a_1+a_2t^2+2\beta^+t}\ \ \text{and}\ \  t_1:=\frac{\int_{\mathbb{R}}|u_{2}|^{4}dx}{\int_{\mathbb{R}}|u_{1}|^{4}dx}.
\end{equation}\\
For any $(u_{1},u_{2})\in\mathcal{M}$, we have
\begin{equation*}
	J_{a_{1},a_{2},\beta}\big(u_{1},u_{2}\big)\geq\kappa_{a_{1},a_{2},\beta^+}(t_{1}).
\end{equation*}
This, combined with  the definition of $\Gamma(a_1,a_2,\beta)$, gives that
\begin{equation}\label{eq3.2}
	\Gamma(a_1,a_2,\beta)=\inf\left\{J_{a_1,a_2,\beta}\big(u_1,u_2\big):(u_1,u_2)\in\mathcal{M}\right\}\geq\inf_{t\in(0,\infty)}\kappa_{a_1,a_2,\beta^+}(t).
\end{equation}
Then, for any $t\in(0,\infty)$,  it is known  that $\kappa_{a_1,a_2,\beta^+}(t)>1$, given that  $0<a_{1}, a_{2}<a^{*}$ and $\beta<\beta_*=\sqrt{(a^*-a_1)(a^*-a_2)}$. A direct computation  also  shows that
\begin{equation*}
	\displaystyle\lim_{  t \rightarrow 0^+ }\kappa_{a_1,a_2,\beta^+}(t)=\frac{a^*}{a_1}>1 \ \ \text{and}\ \  \displaystyle\lim_{  t \rightarrow \infty }\kappa_{a_1,a_2,\beta^+}(t)=\frac{a^*}{a_2}>1.
\end{equation*}
Due to the continuity of $ \kappa _{a_1, a_2, \beta^+ }( t)$, it can thus be concluded that 
\begin{equation*}
	 \mathop{\mathrm{inf\mathrm{} } }\limits _{t\in(0,\infty )}\kappa _{a_1, a_2, \beta^+ }(t) > 1,
\end{equation*}
 which together with (\ref{eq3.2}) yields that $\Gamma(a_1,a_2,\beta)>1$. Therefore, by Theorem \ref{theorem 2.4} (i), problem (\ref{problem 1.7}) admits a  minimizer.\\

\textbf{(ii)}:  Now, assume  that   $a_1>a^*$.\text{ Choosing  } $0\leq\nu\in C_0^\infty(\mathbb{R})\text{ satisfying }\int_{\mathbb{R}}|\nu|^2dx$
\noindent$=1$, it  then follows from  the definition of (\ref{eq2.33}), (\ref{eq2.48}) and (\ref{eq2.49}) that, for any $\beta\in\mathbb{R}$,
\begin{equation}\label{eq3.3}
\begin{split}
		J_{a_1,a_2,\beta}\big(\psi_R,\nu\big)& =\frac{2\|(-\Delta)^{\frac{1}{4}}\psi_R\|_{2}^{2}+2\|(-\Delta)^{\frac{1}{4}}\nu\|_{2}^{2}}{a_{1}\int_{\mathbb{R}}|\psi_R|^{4}dx+a_{2}\int_{\mathbb{R}}|\nu|^4dx+2\beta^{+}\int_{\mathbb{R}}|\psi_R|^2|\nu|^{2}dx} \\
	&\leq\frac{2R+2\|(-\Delta)^{\frac14}\nu\|_2^2+O(R^{-\frac32})}{2\frac{a_1}{a^*}R+a_2\int_{\mathbb{R}}|\nu|^4dx-O(R^{-6})}\to\frac{a^*}{a_1}<1\quad\mathrm{as}\  R\to+\infty,
\end{split}
\end{equation}
which shows that $\Gamma(a_1,a_2,\beta)<1$. Hence,   problem (\ref{problem 1.7}) has no minimizer by Theorem \ref{theorem 2.4} (ii). For the case  $ a_2>a^*$, we can obtain the non-existence of minimizers for problem (\ref{problem 1.7})  by  almost the same proof as the above case.

Finally, suppose that $\beta>\beta^*=\frac{a^*-a_1}2+\frac{a^*-a_2}2$. For $\psi_R$ given by (\ref{eq2.33}), it follows from (\ref{eq2.48}) and (\ref{eq2.49}) that
\begin{align}
	&\int_{\mathbb{R}}\psi_R\sqrt{-\Delta}\psi_R dx-\frac{a_i}{2}\int_{\mathbb{R}}|\psi_R|^{4}dx\notag \\
	&\leq\frac{R}{\|Q\|_{2}^{2}}\|(-\Delta)^{\frac{1}{4}}Q\|_{2}^{2}+O(R^{-\frac{3}{2}})-\frac{a_{i}R}{2\|Q\|_{2}^{4}}\int_{\mathbb{R}}Q^{4}dx+O(R^{-6}) \\
	&=R(1-\frac{a_{i}}{||Q||_{2}^{2}})+O(R^{-\frac{3}{2}})+O(R^{-6}) \ \  i=1,2, \notag
\end{align}
which,  together  with (\ref{eq2.50}), shows that
\begin{equation}\label{eq3.5}
\begin{split}
		 &E_{a_1,a_2,\beta}(\psi_R,\psi_R)\\
	&\leq R\Big(2-\frac{a_1+a_2+2\beta}{\|Q\|_2^2}\Big)+\sum_{i=1}^2V_i(x_0)+O(R^{-\frac32})+O(R^{-6})\stackrel{R}{\to}-\infty.
\end{split}
\end{equation}
Therefore, problem (\ref{problem 1.7}) has no minimizer.\\

\noindent\textbf{Proof of Theorem \ref{theorem1.2}.} Since $0<a_{1}\neq a_{2} < a^{*}$, it is easy to see that $\beta_{* }<\beta^{*}$. From Lemma \ref{lemma2.3}, we know that
\begin{equation*}
	\Gamma\big(a_1,a_2,\beta_*\big)>1\ \ \text{and} \ \ \Gamma\big(a_1,a_2,\beta^*\big)\leq1.
\end{equation*}
Furthermore, Lemma \ref{lemma2.2} states that $\Gamma\big(a_1,a_2,\beta\big)$ is locally Lipschitz continuous. Therefore,  there exists $\hat{\delta}\in(\beta_{*},\beta^{*}]$ such that $\Gamma\big(a_1,a_2,\beta\big)>1$  for any $\beta\in[\beta_{*},\hat{\delta})$. Consequently,  problem (\ref{problem 1.7}) has at least a minimizer for $\beta\in[\beta_{*},\hat{\delta})$ by Theorem \ref{theorem 2.4} (i).\\

Finally, we give the proof of Theorem \ref{theorem1.3}. It is easy to see from the  assumptions of Theorem \ref{theorem1.3}  that the set $\{(a_1,a_2,\beta)\in\mathbb{R}^3:(a_1,a_2,\beta)=(a^*-\beta,a^*-\beta,\beta)$ and $\beta\in(0,a^*)\}$ forms a segment. Moreover, $E_{a_{1},a_{2},\beta}(u_{1},u_{2})$ can be rewritten as
\begin{align}\label{eq3.6}
	E_{a_{1},a_{2},\beta}(u_{1},u_{2})&=\sum_{i=1}^2\int_{\mathbb{R}}u_i\sqrt{-\Delta}u_idx\!+\!\sum_{i=1}^2\int_{\mathbb{R}}V_i(x)|u_i|^2dx \!-\!\frac{a^*}2\int_{\mathbb{R}}(|u_1|^4+|u_2|^4)dx\notag\\
	&\quad+\frac\beta2\int_{\mathbb{R}}(|u_1|^4+|u_2|^4-2|u_1|^2|u_2|^2)dx.
\end{align}
From the Gagliardo-Nirenberg inequality (\ref{eq1.10}),  we have
\begin{align}
	\|(-\Delta)^{\frac{1}{4}}u_{i}\|^{2}-\frac{a^{*}}{2}\int_{\mathbb{R}}|u_i|^4dx&\geq0, \ \ i=1,2,\\\frac{\beta}{2}\int_{\mathbb{R}}(|u_1|^4+|u_2|^4-2|u_1|^2|u_2|^2)dx&\geq0, \ \ \mathrm{for}\ \beta>0,
\end{align}
which  yields that
\begin{equation}\label{eq3.7}
\begin{split}
		&E_{a_{1},a_{2},\beta}(u_{1},u_{2})\\&=\int_{\mathbb{R}}u_1\sqrt{-\Delta}u_1dx+\int_{\mathbb{R}}u_2\sqrt{-\Delta}u_2dx+\sum_{i=1}^2\int_{\mathbb{R}}V_i(x)|u_i|^2dx \\
	& \quad-\frac{a^*}2\int_{\mathbb{R}}(|u_1|^4+|u_2|^4)dx+\frac\beta2\int_{\mathbb{R}}(|u_1|^4+|u_2|^4-2|u_1|^2|u_2|^2)dx \\
	&\ge 0. 
\end{split}
\end{equation}
Combining (\ref{eq3.5}) and (\ref{eq3.7}), since $x_0\!\in \! \mathbb{{R}}$ is arbitrary and $(a_1+a_2+2\beta)/||Q||^2_2=2$ in (\ref{eq3.5}), it follows that
\begin{equation}\label{eq3.8}
	0\le \hat{e}(a^{*}-\beta,a^{*}-\beta,\beta)\le \mathop{\mathrm{inf\mathrm{} } }\limits _{x\in \mathbb{R} }\big(V_{1}(x)+V_{2}(x)\big).
\end{equation}

\noindent\textbf{Proof of Theorem \ref{theorem1.3}. (i)} In order to prove part (i), we assume that
\begin{equation*}
	\mathop{\mathrm{inf\mathrm{} } }\limits _{x\in \mathbb{R} }\big(V_{1}(x)+V_{2}(x)\big)=0\ \ \text{and}\ \   V_i(x)\ge0,\ \ i=1,2,
\end{equation*}
which implies that there exists $x_0\in \mathbb{{R}}$ such that $V_1(x_0)=V_2(x_0)=0$. By (\ref{eq3.8}), we have
\begin{equation}
	0\leq\hat{e}(a^*-\beta,a^*-\beta,\beta)\leq\inf_{x\in\mathbb{R}}\left(V_1(x)+V_2(x)\right)=V_1\big(x_0)+V_2(x_0\big)=0,
\end{equation}
which shows that $\hat{e}(a^*-\beta,a^*-\beta,\beta)=0$. Now, assuming  that $\hat{e}(a^*-\beta,a^*-\beta,\beta)=0$ has a minimizer $(\tilde{u}_1,\tilde{u}_2)\in\mathcal{M}$,  we then obtain  from (\ref{eq3.7}) that
\begin{equation}\label{eq3.10}
	\int_{\mathbb{R}}|\tilde{u}_{1}|^{4}dx=\int_{\mathbb{R}}|\tilde{u}_{2}|^{4}dx\
	\ \text{and}\ \ \int_{\mathbb{R}}|\tilde{u}_{1}|^{2}|\tilde{u}_{2}|^{2}dx=\Big(\int_{\mathbb{R}}|\tilde{u}_{1}|^{4}dx\int_{\mathbb{R}}|\tilde{u}_{2}|^{4}dx\Big)^{\frac{1}{2}},
\end{equation}
\begin{equation}\label{eq3.11}
	\|(-\Delta)^{\frac14}\tilde{u}_1\|_2^2=\frac{a^*}2\int_{\mathbb{R}}|\tilde{u}_{1}|^{4}dx\ \ \mathrm{and}\ \ \int_{\mathbb{R}}V_1(x)|\tilde{u}_1|^2dx=0.
\end{equation}
Since $\tilde{u}_i$ can be assumed to be nonnegative, it can thus be deduced from  the H\"older inequality and (\ref{eq3.10}) that
$\tilde{u}_1(x)\equiv\tilde{u}_2(x)\ge0$ in $\mathbb{{R}}$. This leads to a  contradiction,  as the first equality of (\ref{eq3.11}) implies
that $\tilde{u}_1(x)$ is equal to (up to a translation) $Q(x)$, whereas  the second equality of (\ref{eq3.11}) yields that
$\tilde{u}_1(x)$ has compact support in $\mathbb{{R}}$.

\textbf{(ii):}For $\mathcal{M}$ and $ \mathcal{X}$ defined by (\ref{problem 1.7}), denote
\begin{equation*}
	d(\vec{u},\vec{v}):=\|\vec{u}-\vec{v}\|_\mathcal{X},\ \ \vec{u}, \vec{v}\in\mathcal{M},
\end{equation*}
with
\begin{equation*}
	\|\vec{u}\|_\mathcal{X}=\left(\|u_1\|_{\mathcal{H}_1}^2+\|u_2\|_{\mathcal{H}_2}^2\right)^\frac12, \ \ \vec{u}=(u_1,u_2)\in\mathcal{X},\ \ \vec{v}=(v_1,v_2)\in\mathcal{X}.
\end{equation*}
It is easy to verify that $(\mathcal{M},d)$ forms a complete metric space. Hence, it follows from Ekeland's variational principle \cite[Theorem 5.1]{variational principle} that there exists a minimizing sequence $\{\vec{u}_n=(u_{1n},u_{2n})\}\subset\mathcal{M}$ for $\hat{e}(a^{*}-\beta,a^{*}-\beta,\beta)$ such that
\begin{align}
	\label{eq3.12}\hat{e}(a^*-\beta,a^*-\beta,\beta)\leq E_{a^*-\beta,a^*-\beta,\beta}(\vec{u}_n)\leq\hat{e}(a^*-\beta,a^*-\beta,\beta)+\frac1n,\\E_{a^*-\beta,a^*-\beta,\beta}(\vec{v})\geq \label{eq3.13}E_{a^*-\beta,a^*-\beta,\beta}(\vec{u}_n)-\frac1n\|\vec{u}_n-\vec{v}\|_{\mathcal{X}}\ \ \mathrm{for}\ \ \vec{v}\in\mathcal{M}.
\end{align}
By the compactness of Lemma \ref{lemma2.1}, it is suffices to demonstrate that $\{\vec{u}_n\!=\!(\!u_{1n},u_{2n}\!)\}$\ is bounded in $\mathcal{X}$ in order to prove that $\hat{e}(a^*-\beta,a^*-\beta,\beta)$  is attained. Indeed, if $\{\vec{u}_n=(u_{1n},u_{2n})\}$ is unbounded in $\mathcal{X}$, then there exists a  subsequence of $\{\vec {u}_n\}$, still denoted by $\{\vec {u}_n\}$, such that  $\|\vec{u}_n\|_{\mathcal{X}}\overset{n}{\operatorname*{\to}}+\infty$. By applying  the Gagliardo-Nirenberg inequality (\ref{eq1.10}), (\ref{eq3.7}) and (\ref{eq3.12}), we can infer that
\begin{equation}\label{eq3.14}
	\sum_{i=1}^2\int_{\mathbb{R}}V_i(x)|u_{in}|^2dx\leq E_{a^*-\beta,a^*-\beta,\beta}(\vec{u}_n)\leq\hat{e}(a^*-\beta,a^*-\beta,\beta)+\frac1n.
\end{equation}
It then follows from the definition  (\ref{eq1.8}) and $\|\vec{u}_n\|_{\mathcal{X}}\overset{n}{\operatorname*{\to}}+\infty$ that
\begin{equation}\label{eq3.15}
	\|(-\Delta)^{\frac14}u_{1n}\|_2^2+\|(-\Delta)^{\frac14}u_{2n}\|_2^2\overset{n}{\to}\infty.
\end{equation}
We now introduce the notation   $f_n\approx g_n$, which  represents two function sequences  satisfying $f_n/g_n\to1\mathrm{~as~}n\to\infty$. We now claim that, for $i=1, 2$,
\begin{align}
	\label{eq3.16}\int_{\mathbb{R}}|(-\Delta)^{\frac{1}{4}}u_{in}|^{2}dx&\approx\frac{a^*}{2}\int_{\mathbb{R}}|u_{in}|^{4}dx\xrightarrow{n}+\infty, \\
	\label{eq3.17}\int_{\mathbb{R}}|u_{1n}|^{4}dx&\approx\int_{\mathbb{R}}|u_{2n}|^{4}dx,\\
	\label{eq3.18}\int_{\mathbb{R}}|u_{1n}|^{2}|u_{2n}|^{2}dx&\approx\int_{\mathbb{R}}|u_{2n}|^{4}dx.
\end{align}
By (\ref{eq3.15}), without loss of generality, we may assume that $	\int_{\mathbb{R}}|(-\Delta)^{\frac{1}{4}}u_{1n}|^{2}dx\xrightarrow{n}\infty $.
It can be deduced from (\ref{eq3.12}) that, for $n$ large enough,
\begin{equation}\label{eq3.19}
	0\leq\int_{\mathbb{R}}|(-\Delta)^{\frac{1}{4}}u_{1n}|^{2}dx-\frac{a^*}{2}\int_{\mathbb{R}}|u_{1n}|^{4}dx\leq\hat{e}(a^{*}-\beta,a^{*}-\beta,\beta)+\frac{1}{n}<C,
\end{equation}
where $C$ is a constant.

Multiplying (\ref{eq3.19}) by $1\bigg/\int_{\mathbb{R}}|(-\Delta)^{\frac{1}{4}}u_{1n}|^{2}dx$, we have
\begin{equation}
\begin{split}
		0&\leq1-\frac{a^{*}}{2}\int_{\mathbb{R}}|u_{1n}|^{4}dx\bigg/\int_{\mathbb{R}}|(-\Delta)^{\frac{1}{4}}u_{1n}|^{2}dx\\&\leq C\bigg/\int_{\mathbb{R}}|(-\Delta)^{\frac{1}{4}}u_{1n}|^{2}dx\to0,\ \ \text{as}\ \  n\to\infty,
\end{split}
\end{equation}
which  implies that
\begin{align}  \label{eq3.21} \frac{a^{*}}{2}\int_{\mathbb{R}}|u_{1n}|^{4}dx\xrightarrow{n}\infty \ \ \text{and} \ \ \int_{\mathbb{R}}|(-\Delta)^{\frac{1}{4}}u_{1n}|^{2}dx\bigg/\{\frac{a^*}{2}\int_{\mathbb{R}}|u_{1n}|^{4}dx\}\xrightarrow{n}1.\end{align}
On the other hand, we derive from the Gagliardo-Nirenberg inequality (\ref{eq1.10}) and  (\ref{eq3.12})   that, for $\beta>0$,
\begin{equation}\label{eq3.22}
	0\leq\frac{\beta}{2}\int_{\mathbb{R}}(|u_{1n}|^{4}+|u_{2n}|^{4}-2|u_{1n}|^{2}|u_{2n}|^{2})dx\leq\hat{e}(a^{*}-\beta,a^{*}-\beta,\beta)+\frac{1}{n}.
\end{equation}
Then, for $n$ large,
\begin{align} \notag &\frac{\beta}{2}\Big|\Big(\int_{\mathbb{R}}|u_{1n}|^{4}dx\Big)^{\frac{1}{2}}-\Big(\int_{\mathbb{R}}|u_{2n}|^{4}dx\Big)^{\frac{1}{2}}\Big|^{2}\\ & \leq\frac{\beta}{2}\int_{\mathbb{R}}(|u_{1n}|^{4}+|u_{2n}|^{4}-2|u_{1n}|^{2}|u_{2n}|^{2})dx\\&\leq\hat{e}(a^*-\beta,a^*-\beta,\beta)+\frac{1}{n}\leq \notag C,\end{align}
which, together with $\int_{\mathbb{R}}|u_{1n}|^{4}dx\xrightarrow{n}\infty$, implies that
\begin{equation}\label{eq3.24}
	\int_{\mathbb{R}}|u_{2n}|^4dx\xrightarrow{n}\infty\ \ \mathrm{and}\ \ \int_{\mathbb{R}}|u_{1n}|^4dx\bigg/\int_{\mathbb{R}}|u_{2n}|^4dx\xrightarrow{n}1.
\end{equation}
Similar to (\ref{eq3.19}), we also obtain that
\begin{equation*}
	0\leq\int_{\mathbb{R}}|(-\Delta)^{\frac{1}{4}}u_{2n}|^{2}dx-\frac{a^*}{2}\int_{\mathbb{R}}|u_{2n}|^{4}dx\leq\hat{e}(a^{*}-\beta,a^{*}-\beta,\beta)+\frac{1}{n}.
\end{equation*}
Putting together this with  (\ref{eq3.24}), we have
\begin{equation}\label{eq3.25}
	\int_{\mathbb{R}}|(-\Delta)^{\frac{1}{4}}u_{2n}|^{2}dx\xrightarrow{n}\infty\ \ \mathrm{and}\ \ \int_{\mathbb{R}}|(-\Delta)^{\frac{1}{4}}u_{2n}|^{2}dx\bigg/\{\frac{a^*}2\int_{\mathbb{R}}|u_{2n}|^4dx\}\xrightarrow{n}1.
\end{equation}
Therefore, (\ref{eq3.16}) and (\ref{eq3.17}) are derived from (\ref{eq3.21}),   (\ref{eq3.24}) and (\ref{eq3.25}).

Next, we turn to complete the proof of   (\ref{eq3.18}).
It follows  from  (\ref{eq3.22}) and $\beta>0$ that,  for some constant $C>0$,
\begin{equation*}
	0\leq\int_{\mathbb{R}}(|u_{1n}|^{4}+|u_{2n}|^{4}-2|u_{1n}|^{2}|u_{2n}|^{2})dx\leq C.
\end{equation*}
Multiplying this by $1\bigg/\int_{\mathbb{R}}|u_{2n}|^{4}dx$, we have
\begin{equation}\label{eq3.26}
	0\leq1+\frac{\int_{\mathbb{R}}|u_{1n}|^{4}dx}{\int_{\mathbb{R}}|u_{2n}|^{4}dx}-2\frac{\int_{\mathbb{R}}|u_{1n}|^{2}|u_{2n}|^{2}dx}{\int_{\mathbb{R}}|u_{2n}|^{4}dx}\leq\frac{C}{\int_{\mathbb{R}}|u_{2n}|^{4}dx}\xrightarrow{n}0,
\end{equation}
which, together with (\ref{eq3.24}), implies that
\begin{equation*}
	\frac{\int_{\mathbb{R}}|u_{1n}|^{2}|u_{2n}|^{2}dx}{\int_{\mathbb{R}}|u_{2n}|^{4}dx}\xrightarrow{n}1.
\end{equation*}
Hence, this completes the proof of (\ref{eq3.18}).

To conclude the proof of part (ii), obtaining a contradiction under the condition (\ref{eq3.15}) is necessary. To achieve this,  define
\begin{equation*}
	\epsilon_n^{-1}:=\int_{\mathbb{R}}u_{1n}\sqrt{-\Delta}u_{1n}dx\ \ \mathrm{and}\ \ \tilde{w}_{in}(x)=\epsilon_n^{\frac12}u_{in}(\epsilon_nx),\ \ i=1,2.
\end{equation*}
It is easy to verify that $\|\tilde{w}_{in}\|_2=\|u_{in}\|_2=1$ for $i=1,2$.   By (\ref{eq3.16}), we infer that $\epsilon_n\xrightarrow{n}0$.
Furthermore, we claim that there exists a sequence ${{\left \{ y_{\epsilon_n} \right \} }\subset\mathbb{{R}}}$, as well as positive constants $\eta$ and  $R_0$, such that
\begin{equation}\label{eq3.27}
	\operatorname*{lim\inf}_{n\to\infty}\int_{B_{R_0}(y_{\epsilon_n})}|\tilde{w}_{1n}(x)|^2dx\geq\eta>0.
\end{equation}
To argue by contradiction, we assume that (\ref{eq3.27}) fails. Therefore,  passing to a subsequence if necessary, we deduce that,  for any $R>0$,
\begin{equation}\label{eq3.28}
	\lim\limits_{n\to\infty}\sup\limits_{y\in\mathbb{R}}\int_{B_R(y)}|\tilde{w}_{1n}(x)|^2dx=0.
\end{equation}
By applying the  vanishing lemma \cite[Lemma I.1]{vanish}, we deduce that $\tilde{w}_{1n}\xrightarrow{n}0$ strongly in $L^r(\mathbb{{R}})$ with $r\in(2,\infty)$, which then yields that
\begin{equation}\label{eq3.29}
	\int_{\mathbb{R}}{|\tilde{w}_{1n}|^4}dx\ \to0\ \ \mathrm{as}\ \ n\to\infty.
\end{equation}
However, using (\ref{eq3.16}) and the definition of $\epsilon_n$, we have
\begin{equation}\label{eq3.30}
	\int_{\mathbb{R}}|\tilde{w}_{1n}|^4dx=\epsilon_n\int_{\mathbb{R}}|u_{1n}|^4dx\overset{n}{\to}\frac{2}{a^*}\neq0,
\end{equation}
which results in a contradiction with  (\ref{eq3.29}). According to the same procedures as above, we can also obtain that
\begin{equation}\label{eq new}
	\operatorname*{lim\inf}_{n\to\infty}\int_{B_{R_0}(y_{\epsilon_n})}|\tilde{w}_{2n}(x)|^2dx\geq\eta>0.
\end{equation}

Now, define
\begin{equation*}
	w_{in}(x):=\tilde{w}_{in}\big(x+y_{\epsilon_n}\big)=\epsilon_n^{\frac{1}{2}}u_{in}\big(\epsilon_nx+\epsilon_ny_{\epsilon_n}\big),\ \ i=1,2,
\end{equation*}
where  $\left \{ {y_{\epsilon_n}} \right \} $  is the same as in (\ref{eq3.27}).
Then, we have
\begin{equation}\label{eq3.31}
	\operatorname*{lim\inf}_{n\to\infty}\int_{B_{R_0}(0)}|w_{1n}(x)|^2dx\geq\eta>0.
\end{equation}
We hereby  assert that   ${\left \{ \epsilon_ny_{\epsilon_n} \right \} }$ is a bounded sequence in $\mathbb{{R}}$. We argue by contradiction.  If   ${\left \{ \epsilon_ny_{\epsilon_n} \right \} }$ is not bounded,  we  derive from (\ref{eq3.14})  that, for some $C>0$,
\begin{align}\label{eq3.32}\notag
	\int_{\mathbb{R}}V_1(x)|u_{1n}|^2dx& =\int_{\mathbb{R}}V_1(\epsilon_nx+\epsilon_ny_{\epsilon_n})|w_{1n}|^2dx\\&\leq\sum\limits_{i=1}^2\int_{\mathbb{R}}V_i(x)|u_{in}|^2dx \\
	&\leq\hat{e}(a^*-\beta,a^*-\beta,\beta)+\frac1n\leq C.\notag
\end{align}
This is contradicted by the combination of Fatou's lemma and (\ref{eq3.31}), together with $V_1(x)\xrightarrow{|x|\to\infty}+\infty$.

For any $\varphi(x)\in C_0^\infty(\mathbb{R})$, we now define
\begin{equation*}
	\hat{\varphi}(x)=\varphi\big(\frac{x-\epsilon_ny_{\epsilon_n}}{\epsilon_n}\big),\ \ g(\tau,\sigma)=\frac{1}{2}\int_{\mathbb{R}}\big|u_{1n}+\tau u_{1n}+\sigma\hat{\varphi}\big|^2dx,
\end{equation*}
where $g(\tau,\sigma)$ satisfies
\begin{equation*}
	g\left(0,0\right)=\frac12,\ \ \frac{\partial g\left(0,0\right)}{\partial\tau}=\int_{\mathbb{R}}|u_{1n}|^2dx=1\ \ \mathrm{and} \ \ \frac{\partial g\left(0,0\right)}{\partial\sigma}=\int_{\mathbb{R}}u_{1n}\hat{\varphi}dx.
\end{equation*}
Then, by applying the implicit function theorem, we deduce that  there exist a constant  $\delta_{n}>0$ and  $\tau(\sigma)\in C^1\big((-\delta_n,\delta_n),\mathbb{R}\big)$ such that
\begin{equation*}
	\tau(0)=0, \ \ \tau'(0)=-\int_{\mathbb{R}}u_{1n}\hat{\varphi}dx,\ \ \mathrm{and}\ \ g(\tau(\sigma),\sigma)=g(0,0)=\frac{1}{2},
\end{equation*}
which gives that
\begin{equation*}
	\left(u_{1n}+\tau(\sigma)u_{1n}+\sigma\hat{\varphi},u_{2n}\right)\in\mathcal{M},\ \ \sigma\in(-\delta_n,\delta_n).
\end{equation*}
From (\ref{eq3.13}), we have
\begin{align*}
	&E_{a^*-\beta,a^*-\beta,\beta}(u_{1n}+\tau(\sigma)u_{1n}+\sigma\hat{\varphi},u_{2n})-E_{a^*-\beta,a^*-\beta,\beta}(u_{1n},u_{2n})\\
	&\geq-\frac1n\|(\tau(\sigma)u_{1n}+\sigma\hat{\varphi},0)\|_{\mathcal{X}} .
\end{align*}
Setting $\sigma\to 0^+$ and $\sigma\to 0^-$, respectively, we obtain that
\begin{equation}\label{eq3.33}
	\left|\langle E_{a^*-\beta,a^*-\beta,\beta}^{\prime}(u_{1n},u_{2n}),(\tau^{\prime}(0)u_{1n}+\hat{\varphi},0)\rangle\right|\leq\frac1n\|\tau^{\prime}(0)u_{1n}+\hat{\varphi}\|_{\mathcal{H}_1}.
\end{equation}
Based on the definitions of $\epsilon_n$ and $w_{1n}$, direct calculations give that
\begin{equation}\label{eq3.34}
	\tau^{\prime}(0)=-\int_{\mathbb{R}}u_{1n}\hat{\varphi}dx=-\epsilon_n^{\frac12}\int_{\mathbb{R}}w_{1n}\varphi dx ,\ \ \|\tau^{\prime}(0)u_{1n}+\hat{\varphi}\|_{\mathcal{H}_1}\leq C\epsilon_n.
\end{equation}
Set
\begin{equation}
	\begin{split}
		\mu_{1n}:=& \langle E_{a^*-\beta,a^*-\beta,\beta}^{'}(u_{1n},u_{2n}),(u_{1n},0)\rangle\!=\!2\int_{\mathbb{R}}u_{1n}\sqrt{-\Delta}u_{1n}dx\\ &+2\int_{\mathbb{R}}V_1(x)u_{1n}^2dx-2a_1\int_{\mathbb{R}}|u_{1n}|^4dx-2\beta\int_\mathbb{R}|u_{1n}|^2|u_{2n}|^2dx.
	\end{split}
\end{equation}
We  derive from (\ref{eq3.16})-(\ref{eq3.18}) and (\ref{eq3.32})  that, for $a_1=a^*-\beta$,
\begin{equation}\label{eq3.36}
	\mu_{1n}\approx-a^*\int_{\mathbb{R}}|u_{1n}|^4dx+2\int_{\mathbb{R}}V_1(x)u_{1n}^2dx,\ \ \mathrm{and} \ \ \mu_{1n}\epsilon_n\overset{n}{\to}-2.
\end{equation}
It thus follows from (\ref{eq3.33}), (\ref{eq3.34})-(\ref{eq3.36}) that
\begin{align}\notag&\left|\langle E_{a^*-\beta,a^*-\beta,\beta}^{\prime}(u_{1n},u_{2n}),(\tau^{\prime}(0)u_{1n}+\hat{\varphi},0)\rangle\right|\\&=\frac{2}{\sqrt{\epsilon_n}}\bigg|\int_{\mathbb{R}}\varphi\sqrt{-\Delta}w_{1n}dx+\epsilon_n\int_{\mathbb{R}}V_1(\epsilon_nx+\epsilon_ny_{\epsilon_n})w_{1n}\varphi dx\!-\!\frac{\mu_{1n}\epsilon_n}2\int_{\mathbb{R}}w_{1n}\varphi \notag dx\!\\ & \quad- \!\int_{\mathbb{R}}(a_1{w_{1n}^3}+\beta{w_{2n}^2{w_{1n}}})\varphi dx\bigg|\\&\leq\frac1n\|\tau^{\prime}(0)u_{1n}+\hat{\varphi}\|_{\mathcal{H}_1}\leq\frac{C\epsilon_n}n\to0\ \ \mathrm{as}\ \ n\to\infty,\notag
\end{align}
which means that
\begin{equation}\label{eq3.38}
\begin{split}
	\int_{\mathbb{R}}\varphi\sqrt{-\Delta} w_{1n}dx+\epsilon_n\int_{\mathbb{R}}V_1(\epsilon_nx+\epsilon_ny_{\epsilon_n})w_{1n}\varphi dx-\frac{\mu_{1n}\epsilon_n}{2}\int_{\mathbb{R}}w_{1n}\varphi dx\\-a_1\int_{\mathbb{R}}w_{1n}^3\varphi dx-\beta\int_{\mathbb{R}}{w_{2n}^2}w_{1n}\varphi dx\xrightarrow{n}0.
\end{split}
\end{equation}
Given that $\{w_{1n}\}$ and  $\{w_{2n}\}$ are bounded in $H^{\frac12}(\mathbb{R})$, one  may suppose that, for $w_1, w_2\in H^{\frac12}(\mathbb{R})$,
\begin{equation*}
	w_{in}\stackrel{n}{\rightharpoonup}w_i\text{ weakly in }H^{\frac12}(\mathbb{R}).
\end{equation*}
Moreover, we know from (\ref{eq3.31}) that $w_{1}\not\equiv0$. In the sequel, we will show  that $ w_{2}\not\equiv0$.

On basis of  (\ref{eq3.36}) and (\ref{eq3.38}), we found that $w_1$ and $w_2$ are weak solutions of the following equation
\begin{equation}\label{weak}
	\sqrt{-\Delta}w_1+w_1=a_1w_1^3+\beta w_2^2 w_1,
\end{equation}
where $ a_1=a^*-\beta$. By noting  $a_2=a_1$, we know that $w_1$ and $w_2$ also  satisfy
\begin{equation*}
	\sqrt{-\Delta}w_2+w_2=a_1w_2^3+\beta w_1^2 w_2,
\end{equation*}
which, together with (\ref{weak}), claims that $(w_1,w_2)$ is a weak solutions for the following system
\begin{equation}\label{eq3.39}
	\begin{cases}\sqrt{-\Delta}u+u=a_1u^3+\beta v^2u,\\\sqrt{-\Delta}v+v=a_1v^3+\beta u^2v.\end{cases}
\end{equation}
Arguing by contradiction, assume that $w_2\equiv0$, which implies that  $w_1$ satisfies the  equation
\begin{equation}\label{eq3.40}
	\sqrt{-\Delta}w_1+w_1=a_1w_1^3.
\end{equation}
According to Lemma 5.4 in \cite{x10}, it is known  that $w_1$ satisfies the following  Pohozaev identity
\begin{equation}\label{t1}
	\int _{\mathbb{R}}|w_1|^2dx =\frac{a_1}{2} \int _{\mathbb{R}}|w_1|^4dx .
\end{equation}
We  also deduce from (\ref{eq3.40}) that
\begin{equation}\label{t2}
	\int_{\mathbb{R}}|(-\Delta)^{\frac14}w_1|^2dx+\int_{\mathbb{R}}|w_1|^2dx=a_1\int_{\mathbb{R}}|w_1|^4dx.
\end{equation}
Combining   (\ref{t1}) and (\ref{t2}), yields that
\begin{equation}\label{t4}
	\int_{\mathbb{R}}|(-\Delta)^{\frac14}w_1|^2dx=\int_{\mathbb{R}}|w_1|^2dx=\frac{a_1}{2}\int_{\mathbb{R}}|w_1|^4dx.
\end{equation}
It then follows from the Gagliardo-Nirenberg inequality  (\ref{eq1.10}) and (\ref{t4})   that
\begin{equation}\label{t3}
	\frac{2}{a_1}\|w_1\|_2^2=\int_{\mathbb{R}}|w_1|^4dx\leq\frac{2}{a^*}\|(-\Delta)^{\frac{1}{4}}w_1\|_2^2\|w_1\|_2^2=\frac{2}{a^*}\|w_1\|_2^4.
\end{equation}
Since $a_1=a^*-\beta \in(0,a^*)$, we then obtain from (\ref{t3}) that
\begin{equation*}
	\|w_1\|_2^2\geq\frac{a^*}{a_1}>1,
\end{equation*}
which contradicts that  $\|w_1\|_2^2\leq\lim\limits_{n\to\infty}\|w_{1n}\|_2^2=1$.  Therefore, $w_2\not\equiv0$, and $(w_1,w_2)$ forms a nontrivial solution of the system (\ref{eq3.39}).

Let
\begin{equation*}
	\hat{\mathcal{X}}=H^{\frac{1}{2}}(\mathbb{R})\times H^{\frac{1}{2}}(\mathbb{R}), \ \ \|u\|_{H^{\frac{1}{2}}}^{2}=\int_{\mathbb{R}}(|(-\Delta)^{\frac{1}{4}}u|^{2}+u^{2})dx\ \ \text{and}
\end{equation*}
\begin{align*}
	\hat{\mathcal{M}}=\Big\{(u,v)\in\hat{\mathcal{X}}:u,v\neq0, 	&\|u\|_{H^{\frac{1}{2}}}^{2}=\int_{\mathbb{R}}(a_{1}|u|^4+\beta|v|^2|u|^{2})dx, \\ &\|v\|_{H^{\frac{1}{2}}}^{2}=\int_{\mathbb{R}}(a_{1}|v|^4+\beta|u|^2|v|^{2})dx\Big\},
\end{align*}
and
\begin{equation}\label{eq3.42}
	\hat{J}\left(u,v\right):=\frac12(\left\|u\right\|_{H^{\frac12}}^2+\left\|v\right\|_{H^{\frac12}}^2)-\frac14\int_{\mathbb{R}}(a_1|u|^4+a_1|v|^4+2\beta|u|^2|v|^2).
\end{equation}
For $Q>0$ being a positive radially symmetric  solution of (\ref{classical eq}), we know by Lemma \ref{app 1} and Lemma  \ref{app2} in the Appendix that
\begin{equation}\label{eq3.43}
	(u_0,v_0)=\begin{cases}(\frac{1}{\sqrt{a^*}}Q(x),\frac{1}{\sqrt{a^*}}Q(x)),&a_1\neq\beta\ \ \text{and} \ \ a_1+\beta=a^*,\\[2ex](\frac{1}{\sqrt{a_1}}Q\sin\theta,\frac{1}{\sqrt{a_1}}Q\cos\theta),&\theta\in(0,2\pi)\ \ \text{and} \ \ a_1=\beta=\frac{a^*}{2},\end{cases}
\end{equation}
is a minimizer for the following problem
\begin{equation*}
	\inf\{\hat{J}(u,v):(u,v)\in	\hat{\mathcal{M}}\}.
\end{equation*}
According to   Lemma 5.4 in \cite{x10}, a solution $(u,v)$ of system (\ref{eq3.39}) satisfies the following  Pohozaev identity
\begin{equation}\label{eq3.44}
	\int_{\mathbb{R}}|u|^2+|v|^2dx=\frac12\int\limits_{\mathbb{R}^3}(a_1|u|^4+a_1|v|^4+2\beta|u|^2|v|^2)dx,
\end{equation}
which,  combined with $(u,v)\in 	\hat{\mathcal{M}}$, shows that
\begin{equation}\label{eq3.45}
	\begin{split}
			\int_{\mathbb{R}}|(-\Delta)^\frac14u|^2+|(-\Delta)^\frac14v|^2dx&=\int_{\mathbb{R}}|u|^2+|v|^2dx\\
			&=\frac12\int_{\mathbb{R}}(a_1|u|^4+a_1|v|^4+2\beta|u|^2|v|^2)dx.
	\end{split}
\end{equation}
Combining  (\ref{eq3.42}), (\ref{eq3.43}) and (\ref{eq3.45}) gives that
\begin{equation}\label{eq3.46}
	\hat{J}(w_1,w_2)=\frac{1}{2}\int_{\mathbb{R}}|w_1|^2+|w_2|^2dx\geq \hat{J}(u_0,v_0)=\frac{1}{2}\int_{\mathbb{R}}|u_0|^2+|v_0|^2dx=1.
\end{equation}
Since $\|w_{in}\|_2^2\equiv1$ and $w_i\not\equiv0$, we have $0<\|w_i\|_2^2\leq1$, and then (\ref{eq3.46}) implies  that
\begin{equation}\label{eq3.47}
	\|w_i\|_2^2=\int_{\mathbb{R}}|w_i|^2=1, \ \  i=1, 2,
\end{equation}
which means that
\begin{equation}\label{eq3.48}
	w_{in}\overset{n}{\to} w_i\text{ strongly in }L^2(\mathbb{R}),\ \  i=1, 2.
\end{equation}
Given that ${\left \{ \epsilon_ny_{\epsilon_n} \right \} }$ is bounded in $\mathbb{R}$,
one may assume that $\epsilon_ny_{\epsilon_n}\xrightarrow{n}h_0\in\mathbb{R}$. Due to the  Fatou lemma, along with (\ref{eq3.32}), (\ref{eq3.47}) and (\ref{eq3.48}), we can infer that
\begin{align*}
	\hat{e}(a^{*}-\beta,a^{*}-\beta,\beta)& \geq\sum_{i=1}^2\int_{\mathbb{R}}\lim_{n\to\infty}V_i(\epsilon_nx+\epsilon_ny_{\epsilon_n})|w_{in}|^2dx \\
	&=\sum_{i=1}^2\int_{\mathbb{R}}V_i(h_0)|w_i|^2dx=V_1(h_0)+V_2(h_0),
\end{align*}
which contradicts (\ref{eq1.12}). Therefore, $\{(u_{1n},u_{2n})\}$ is bounded in $\mathcal{X}$, thereby completing the proof of Theorem \ref{theorem1.3} (ii).

\section{Appendix }
This appendix focuses on the proofs of the following lemmas, which are dedicated to proving Theorem \ref{theorem1.3} (ii).   We now come to consider  the following system
\begin{equation}\label{system4.1}
	\begin{cases}\sqrt{-\Delta}u+u=au^3+\beta v^2u,\ \ &(u,v)\in\hat{\mathcal{X}}=H^{\frac12}(\mathbb{R})\times H^{\frac12}(\mathbb{R}),\\\sqrt{-\Delta}v+v=av^3+\beta u^2v,\ \ &a,\ \beta\in\mathbb{R}^+.\end{cases}
\end{equation}
The system (\ref{system4.1}) can be described  by the following constrained
minimization problem
\begin{equation}\label{eq4.3}
	\rho _{1}:=\inf\left\{\hat{J}(u,v):(u,v)\in\hat{\mathcal{M}}\right\},
\end{equation}
where
\begin{equation}\label{eq4.2}
	\hat{J}(u,v)=\frac12(\|u\|_{H^{\frac12}}^2+\|v\|_{H^{\frac12}}^2)-\frac14\int_{\mathbb{R}}(a|u|^4+a|v|^4+2\beta|u|^2|v|^2)dx\ \ \text{and}
\end{equation}
\begin{align*}
	\hat{\mathcal{M}}=\Big\{(u,v)\in\hat{\mathcal{X}}:u,v\neq0, 	&\|u\|_{H^{\frac{1}{2}}}^{2}=\int_{\mathbb{R}}(a_{1}|u|^4+\beta|v|^2|u|^{2})dx, \\ &\|v\|_{H^{\frac{1}{2}}}^{2}=\int_{\mathbb{R}}(a_{1}|v|^4+\beta|u|^2|v|^{2})dx\Big\}.
\end{align*}
Here, 	$$\|u\|_{H^{\frac12}}^2=\int_{\mathbb{R}}(|(-\Delta)^{\frac14}u|^2+u^2)dx~\text{and}~\|v\|_{H^{\frac12}}^2=\int_{\mathbb{R}}(|(-\Delta)^{\frac14}v|^2+v^2)dx.$$ Then, the minimizers of (\ref{eq4.3}) are ground state solutions for (\ref{system4.1}).

For any $(u,v)\in \hat{\mathcal{M}}$, we have
\begin{equation*}
	\hat{J}(u,v)=\frac14\|u\|_{H^{\frac12}}^2+\frac14\|v\|_{H^{\frac12}}^2.
\end{equation*}
If $k>0$ and $s>0$  satisfy
\begin{equation}\label{eq4.4}
	ak+\beta s=1,\text{ and }\ as+\beta k=1,
\end{equation}
we  claim that
\begin{equation*}
	\left(u_0,v_0\right):=\left(\sqrt{k}Q,\sqrt{s}Q\right)\in\hat{\mathcal{M}},
\end{equation*}
where $Q>0$ is a  solution of equation (\ref{classical eq}).\\
It can  be derived
from  Pohozaev identity (\ref{eq1.11}) that
\begin{equation*}
	\left\|u_0\right\|_{H^{\frac12}}^2=\|\sqrt{k}Q\|_{H^{\frac12}}^2=2k\left\|Q\right\|_2^2=2ka^*,\ \ \text{and if }  ak+\beta s=1, 
\end{equation*}
\begin{equation*}
	\int_{\mathbb{R}}(a|u_0|^4\!+\!\beta|v_0|^2|u_0|^2)\!=\!(ak^{2}+\beta ks)\int_{\mathbb{R}}|Q|^4dx\!=\!2k(ak+\beta s)a^{*} \\
	\!=\!2ka^{*}\!=\!\|u_{0}\|_{H^{\frac{1}{2}}}^{2}.
\end{equation*}
Similarly, we   have
\begin{equation*}
	\left\|v_{0}\right\|_{H^{\frac12}}^2=\int_{\mathbb{R}}(a|v_0|^4+\beta|u_0|^2|v_0|^2)dx=2sa^*,\ \  \text{if}\  as+\beta k=1.
\end{equation*}
and then $(u_0,v_0)\in \hat{\mathcal{M}}$. It follows from the  definition of (\ref{eq4.3}) that
\begin{equation}\label{eq4.5}
	\rho_1\leq \hat{J}(u_0,v_0)=\frac{1}{4}\|u_0\|_{H^{\frac{1}{2}}}^2+\frac{1}{4}\|v_0\|_{H^{\frac{1}{2}}}^2=\frac{a^*}{2}(k+s).
\end{equation}

\begin{lemma}\label{app 1}
	Assuming that $k>0$, $s>0$ and (\ref{eq4.4}) is satisfied, then $(u_0,v_0)=(\sqrt{k}Q,\sqrt{s}Q)$ is a ground state solution for system (\ref{system4.1}) if $a\neq\beta$.
\end{lemma}
\begin{proof}
	We now let $\{(u_n,v_n)\}\subset\hat{\mathcal{X}}$ be a minimizing sequence for (\ref{eq4.3}), that is,
	\begin{equation}\label{eq4.6}
		\{(u_n,v_n)\}\subset\hat{\mathcal{M}}\ \ \text{and}\ \ \hat{J}(u_n,v_n)=\frac14\|u_n\|_{H^{\frac12}}^2+\frac14\|v_n\|_{H^{\frac12}}^2\stackrel{n\to \infty}{\to} \rho_1.
	\end{equation}
	It follows from the Gagliardo-Nirenberg inequality (\ref{eq1.10}) that, for $u_n\in H^{\frac{1}{2}}(\mathbb{R})$,
	\begin{equation*}
		\int_{\mathbb{R}}|u_{n}|^4dx\leq\frac{2}{a^*}\|(-\Delta)^{\frac{1}{4}}u_n\|_2^2\|u_n\|_2^2\leq\frac{1}{2a^*}\|u_n\|_{H^{\frac{1}{2}}}^4.
	\end{equation*}
	Since $\{(u_n,v_n)\}\in \hat{\mathcal{M}}$,  we have
	\begin{align}
		\label{eq4.7}e_{u_{n}}&:=\frac{1}{\sqrt{2a^{*}}}\Big(\int_{\mathbb{R}}|u_{n}|^4dx\Big)^{\frac{1}{2}} \notag\\ &\leq\frac{1}{2a^{*}}\|u_{n}\|_{H^{\frac{1}{2}}}^{2}=\frac{1}{2a^{*}}\int_{\mathbb{R}}(a|u_{n}|^4+\beta|v_{n}|^2|u_{n}|^2)dx,\\ \label{eq4.8}e_{v_{n}}&:=\frac{1}{\sqrt{2a^{*}}}\Big(\int_{\mathbb{R}}|v_{n}|^4dx\Big)^{\frac{1}{2}}\notag \\ &\leq\frac{1}{2a^{*}}\|v_{n}\|_{H^{\frac{1}{2}}}^{2}=\frac{1}{2a^{*}}\int_{\mathbb{R}}(a|v_{n}|^4+\beta|u_{n}|^2|v_{n}|^2)dx.
	\end{align}
	In addition, using the H\"older inequality, we yield that
	\begin{equation}\label{eq4.9}
		\int_{\mathbb{R}}|u_{n}|^2|v_{n}|^2dx\le2a^*e_{u_n}e _{v_n},
	\end{equation}
	which, together with (\ref{eq4.7}) and (\ref{eq4.8}), implies that
	\begin{equation}\label{eq4.10}
		\begin{cases}
			2a^{*}e_{u_{n}}\leq\|u_{n}\|_{H^{\frac{1}{2}}}^{2}\leq2a^{*}\Big(ae_{u_{n}}^{2}+\beta e_{u_{n}}e_{v_{n}}\Big),\\2a^{*}e_{v_{n}}\leq\|v_{n}\|_{H^{\frac{1}{2}}}^{2}\leq2a^{*}\Big(ae_{v_{n}}^{2}+\beta e_{u_{n}}e_{v_{n}}\Big).
		\end{cases}
	\end{equation}
	Combining  (\ref{eq4.5}), (\ref{eq4.6}) and (\ref{eq4.10}), we infer that
	\begin{equation}\label{eq4.11}
		2a^*(e_{u_n}+e_{v_n})\leq\|u_n\|_{H^{\frac12}}^2+\|v_n\|_{H^{\frac12}}^2=4\rho_1+o(1)\leq2a^*(k+s)+o(1).
	\end{equation}
	This and (\ref{eq4.10}) give that
	\begin{equation}\label{eq4.12}
		\begin{cases}e_{u_n}+e_{v_n}\leq k+s+o(1),\\1\leq ae_{u_n}+\beta e_{v_n},\\1\leq\beta e_{u_n}+ae_{v_n}.\end{cases}
	\end{equation}
	We now define $\eta _{u_{n}}=e_{u_{n}}-k\mathrm{~and~}\eta_{v_{n}}=e_{v_{n}}-s$.  By (\ref{eq4.4}), we can rewrite (\ref{eq4.12}) as
	\begin{equation}\label{eq4.13}
		\begin{cases}\eta _{u_n}+\eta _{v_n}\leq o(1),\\a\eta _{u_n}+\beta\eta _{v_n}\geq0,\\\beta\eta _{u_n}+a\eta _{v_n}\geq0.\end{cases}
	\end{equation}
	Set
	\begin{equation*}
		\vartheta    _n=\Big\{\big(\eta_{u_n},\eta_{v_n}\big):\eta_{u_n}+\eta_{v_n}\leq o(1),a\eta_{u_n}+\beta\eta_{v_n}\geq0,\beta\eta_{u_n}+a\eta_{v_n}\geq0\Big\}.
	\end{equation*}
	It is easy to see that, for each $n\ge1$,  $\vartheta_n $ forms a triangular region with the vertex $(0,0)$,  and its diameter goes to 0 as $n\to +\infty$ if $a\neq\beta$, which then implies that
	\begin{equation*}
		\eta_{u_n}=e_{u_n}-k\xrightarrow{n}0,\ \eta_{v_n}=e_{v_n}-s\xrightarrow{n}0,
	\end{equation*}
	that is, $e_{u_n}\xrightarrow{n\to \infty} k, e_{v_n}\xrightarrow{n\to \infty} s.$
	This together with  (\ref{eq4.11}) give that,
	\begin{equation*}
		\rho_1=\frac{a^*}{2}(k+s).
	\end{equation*}
	Then, it follows from (\ref{eq4.5}) that $\hat{J}(u_0,v_0)=\frac{a^*}{2}(k+s)=\rho_1$, which implies that $(u_0,v_0)=\begin{pmatrix}\sqrt{k}Q,\sqrt{s}Q\end{pmatrix}$ is a ground state solution.
	
	\begin{lemma}\label{app2}
		For any $\theta\in(0,2\pi)$, $\left(u_0,v_0\right)=\left(\frac1{\sqrt{a}}Q\sin\theta,\frac1{\sqrt{a}}Q\cos\theta\right)$ is a ground state solution for (\ref{system4.1}) with $a=\beta$.
	\end{lemma}
	\noindent\textbf{Proof.} We now define
	\begin{equation}\label{eq4.15}
		P \left(u_1,u_2\right): =\frac{\int_{\mathbb{R}}\Big(|(-\Delta)^{\frac{1}{4}}u_1|^2+|(-\Delta)^{\frac{1}{4}}u_2|^2\Big)dx\int_{\mathbb{R}}\Big(|u_1|^2+|u_2|^2\Big)dx}{a\int_{\mathbb{R}}\Big(|u_1|^4+|u_2|^4+2|u_1|^2|u_2|^2\Big)dx},
	\end{equation}
	and consider the following minimization problem
	\begin{equation}\label{contra}
		p:=\inf_{(0,0)\neq(u_1,u_2)\in H^{\frac{1}{2}}(\mathbb{R})\times H^{\frac{1}{2}}(\mathbb{R})}P(u_1,u_2).
	\end{equation}
	Choosing $(\frac{1}{\sqrt{2a}}Q,\frac{1}{\sqrt{2a}}Q)$ as the test function, we then have
	\begin{equation}\label{eq4.16}
		p\leq P\Big(\frac1{\sqrt{2a}}Q,\frac1{\sqrt{2a}}Q\Big)=\frac{\|Q\|_2^2}{2a}=\frac{a^*}{2a}.
	\end{equation}\\
	It follows from Theorem 7.8 and Theorem 7.13 in \cite{x19} and the Gagliardo-Nirenberg inequality (\ref{eq1.10}) that, for any $(u_1,u_2)\in H^{\frac{1}{2}}(\mathbb{R})\times H^{\frac{1}{2}}(\mathbb{R})$,
	\begin{align}\label{eq4.17}
		P\left(u_1,u_2\right)& =\frac{\int_{\mathbb{R}}\Big(|(-\Delta)^{\frac{1}{4}}u_1|^2+|(-\Delta)^{\frac{1}{4}}u_2|^2\Big)dx\int_{\mathbb{R}}\Big(|u_1|^2+|u_2|^2\Big)dx}{a\int_{\mathbb{R}}\Big(|u_1|^4+|u_2|^4+2|u_1|^2|u_2|^2\Big)dx} \notag\\
		&\geq\frac{\int_{\mathbb{R}}\left|(-\Delta)^{\frac14}\sqrt{|u_1|^2+|u_2|^2}\right|^2dx\int_{\mathbb{R}}\left(\sqrt{|u_1|^2+|u_2|^2}\right)^2dx}{a\int _{\mathbb{R}}\Big(\sqrt{|u_1|^2+|u_2|^2} \Big)^4dx}\\&\geq\frac{\|Q\|_2^2}{2a},\notag
	\end{align}
	which, together with  (\ref{eq4.16}), implies that $p=\frac{\|Q\|_2^2}{2a}=\frac{a^*}{2a}$. Therefore, for any $\theta\in(0,2\pi)$,
	\begin{equation}\label{eq4.18}
		\Big(\frac{1}{\sqrt{a}}Q(x)\sin\theta,\frac{1}{\sqrt{a}}Q(x)\cos\theta\Big)
	\end{equation}
	is an optimizer of (\ref{contra}). Furthermore, it is easy to verify that if $\left ( u_1, u_2\right )$ is an optimizer of (\ref{contra}), then, up to scalings, $(u_1,u_2)$ also constitutes a ground state of (\ref{system4.1}). Consequently, if $(u_1,u_2)$ is the form of (\ref{eq4.18}),   it is a ground state of (\ref{system4.1}). \renewcommand{\qedsymbol}{}

\end{proof}







\section*{Acknowledgments}  The authors would like to express their sincere thanks to
the Editor and the Reviewers for the constructive comments to improve this paper.

\medskip
Received November 2024; 1st revision March 2025; 2nd revision April 2025; early access April 2025.
\medskip

\end{document}